\documentclass{amsart}

\usepackage{amsfonts}
\usepackage{amssymb}
\usepackage{mathtools}
\usepackage{tikz-cd}
\usepackage{amsmath}
\usepackage{comment}
\usepackage{caption}

\usepackage{tikz}
\usetikzlibrary{arrows,decorations.pathreplacing}
\usetikzlibrary{shapes}%,snakes}
\usetikzlibrary{arrows,calc}
\usepackage{thm-restate}

\newcommand\cD{\mathcal{D}}
\usepackage[
	pdftitle={},
	pdfauthor={},
	ocgcolorlinks,
	linkcolor=linkblue,
	citecolor=linkred,
	urlcolor=linkred]
{hyperref}
\usepackage{color}
\definecolor{linkred}{rgb}{0.75,0,0}
\definecolor{linkblue}{rgb}{0,0,0.75}
\usepackage{multicol}
\usepackage{multirow}
\usepackage{enumerate}

\theoremstyle{plain}
\newtheorem{theorem}{Theorem}

\newtheorem{proposition}{Proposition}[section]
\newtheorem{lemma}[proposition]{Lemma}
\newtheorem{thrm}[proposition]{Theorem}

\newtheorem{corollary}[proposition]{Corollary}

\newcommand{\bt}{\begin{theorem}}
\newcommand{\et}{\end{theorem}}

\theoremstyle{definition}
\newtheorem{definition}[proposition]{Definition}
\newtheorem{remark}[proposition]{Remark}
\newcommand{\beq}{\begin{equation}}
\newcommand{\eeq}{\end{equation}}

\newcommand{\cal}{\mathcal}

\newcommand{\cb}{\mathcal{B}}
\newcommand{\cc}{\mathcal{C}}
\newcommand{\cd}{\mathcal{D}}

\newcommand{\cp}{{\cal P}}

\newcommand{\bp}{\mathbb{P}}
\newcommand{\bq}{\mathbb{Q}}
\newcommand{\br}{\mathbb{R}}
\newcommand{\bz}{\mathbb{Z}}

\newcommand{\modm}{\cal M}

\newcommand{\wc}{\mathfrak{wc}}

\newcommand{\Mbar}[2]{\overline{\mathcal{M}}_{{#1}, {#2}}}

\newtheorem{innercustomthm}{Theorem}
\newenvironment{customthm}[1]
  {\renewcommand\theinnercustomthm{#1}\innercustomthm}
  {\endinnercustomthm}

\newcommand\bbC{\mathbb{C}}

\newcommand\bbR{\mathbb{R}}

\newcommand\bbZ{\mathbb{Z}}

\newcommand\cA{\mathcal{A}}
\newcommand\cB{\mathcal{B}}
\newcommand\cC{\mathcal{C}}

\newcommand\cL{\mathcal{L}}
\newcommand\cM{\mathcal{M}}

\newcommand\cP{\mathcal{P}}

\newcommand\cV{\mathcal{V}}

\renewcommand{\b}{\boldsymbol}

\begin{document}
	
\title{Weil-Petersson volumes, stability conditions and wall-crossing}

\author{Lukas Anagnostou}
\address{School of Mathematics and Statistics, University of Melbourne, VIC 3010, Australia}
\email{\href{mailto:lanagnostou@student.unimelb.edu.au}{lanagnostou at student.unimelb.edu.au}}
\email{\href{mailto:norbury@unimelb.edu.au}{norbury at unimelb.edu.au}}

\author{Scott Mullane}
\address{Humboldt-Universit\"at zu Berlin, Institut f\"ur Mathematik,  Unter den Linden 6
\hfill \newline\texttt{}
10099 Berlin, Germany and School of Mathematics and Statistics, University of Melbourne, VIC 3010, Australia}
\email{\href{mailto:mullanes@unimelb.edu.au}{mullanes at unimelb.edu.au}}

\author{Paul Norbury}
%\address{School of Mathematics and Statistics, University of Melbourne, VIC 3010, Australia}

\thanks{}
\date{\today}

\begin{abstract}
In this paper we study Weil-Petersson volumes of the moduli spaces of conical hyperbolic surfaces.  The moduli spaces are parametrised by their cone angles which naturally live inside Hassett's space of stability conditions on nodal curves.  Such stability conditions  produce weighted pointed stable curves which define compactifications of the moduli space of curves generalising the Deligne-Mumford compactification.   The space of stability conditions decompose into chambers separated by walls.  We assign to each chamber a polynomial corresponding to the Weil-Petersson volume of a moduli space of conical hyperbolic surfaces.   The chambers are naturally partially ordered and the maximal chamber is assigned Mirzakhani's polynomial.  We calculate wall-crossing polynomials, which relates the polynomial on any chamber to Mirzakhani's polynomial via wall-crossings,
and we show how to apply this in particular cases.  Since the polynomials are volumes, they have nice properties such as positivity, continuity across walls, and vanishing in certain limits.  
\end{abstract}

\maketitle

\addtocontents{toc}{\protect\setcounter{tocdepth}{1}}

\tableofcontents

\section{Introduction}

The Weil-Petersson symplectic form $\omega^{WP}$ is a K\"ahler form defined on the moduli space $\modm_{g,n}$ of genus $g$ curves with $n$ marked points.  It defines a finite measure, hence a well-defined volume of the moduli space.  To study the volume of $\modm_{g,n}$, Mirzakhani \cite{MirSim} used a family of non-K\"ahler deformations of $\omega^{WP}$, defined by deforming complete hyperbolic surfaces to hyperbolic surfaces with geodesic boundary.  The new symplectic structures produce a family of volumes which Mirzakhani proved are given by polynomials in the lengths of the geodesic boundary.  In this paper we consider K\"ahler deformations of $\omega^{WP}$, defined in \cite{STrWei} by deforming complete hyperbolic surfaces to conical hyperbolic surfaces.  This produces  a family of volumes which we prove to be piecewise polynomial, depending on a chamber structure arising out of stability conditions due to Hassett \cite{HasMod}.

The moduli space of hyperbolic surfaces with labeled boundary components is defined as
\begin{align*}
\modm_{g,n}^{\text{hyp}}(L_1&,...,L_n)=\Big\{(\Sigma,\beta_1,...,\beta_n)\mid \Sigma \text{ oriented hyperbolic surface},\\
&\text{ genus }\Sigma=g, \partial \Sigma=\sqcup\beta_j, 2\cosh(L_j/2)=\left|\text{tr}(A_j)\right|\Big\}/\sim
\end{align*} 
where $A_j\in PSL(2,\br)$ represents the conjugacy class defined by the holonomy of the metric around the boundary component $\beta_j$.  The quotient is by isometries preserving each $\beta_j$ (and necessarily acting on $\beta_j$ by a non-trivial rotation).  We consider only $L_j\in\br_+\cup\{0\}\cup i\hspace{.5mm} \br_+$ %corresponding to a hyperbolic, parabolic, respectively elliptic, conjugacy class, 
where the boundary component is geometrically realised by a geodesic, a cusp, respectively a cone angle $\theta_j=-iL_j\in\br^+$. %which is an incomplete hyperbolic metric on the punctured surface. 

The moduli space $\modm_{g,n}^{\text{hyp}}(L_1,...,L_n)$ comes equipped with a finite measure which can be constructed in different ways.  It arises as the top power of a natural symplectic form $\omega^{WP}_\mathbf{L}$, defined via the Weil-Petersson construction of a Hermitian metric or via a natural symplectic form on the character variety \cite{GolInv}, or as the Reidemeister torsion of the natural complex defined by the associated flat $PSL(2,\br)$ connection \cite{WitQua}, or via Fenchel-Nielsen coordinates when they exist \cite{WolWei}.   The total measure is defined to be the volume of the moduli space:
\begin{equation}  \label{eq:vol}
\mathrm{Vol}\big(\cM^{\text{hyp}}_{g,n}(\mathbf{L})\big)=\int_{\modm_{g,n}^{\text{hyp}}(\mathbf{L})}\exp\omega^{WP}_\mathbf{L}.
\end{equation} 
When $\mathbf{L}=(L_1,...,L_n)\in\br_{\geq 0}^n$, via recursion relations between volumes, Mirzakhani \cite{MirSim} proved that the volume is given by a symmetric polynomial in $L_1,...,L_n$
\[V^{\mathrm{Mirz}}_{g,n}(\mathbf{L}):=\text{Vol}(\modm_{g,n}^{\text{hyp}}(\mathbf{L}))\in\br[\theta_1,...,\theta_n],\quad \mathbf{L}\in\br_{\geq 0}^n.\] 

This paper extends Mirzakhani's results by calculating volumes of $\modm_{g,n}^{\text{hyp}}(\mathbf{L})$ for imaginary $\mathbf{L}\in i[0,2\pi)^n$, and producing relations between these volumes.  Via the existence of conical metrics with prescribed cone angles on any conformal surface, proven in \cite{McOPoi}, there is a natural isomorphism with the moduli space of curves $\modm_{g,n}^{\text{hyp}}(\mathbf{L})\cong\modm_{g,n}$, for any $\mathbf{L}\in i[0,2\pi)^n$.  The isomorphism is used to define a natural K\"ahler metric \cite{STrWei} on $\modm_{g,n}^{\text{hyp}}(\mathbf{L})$ which gives rise to a K\"ahler form $\omega^{WP}_{\mathbf{L}}$ on $\modm_{g,n}^{\text{hyp}}(\mathbf{L})$, with top power a finite measure.  The total measure defines the volume via \eqref{eq:vol}.  For small enough angles, specifically
\[\sum_{j=1}^n\theta_j=-i\sum_{j=1}^n L_j<2\pi,\] 
it is proven in \cite{ANoVol} that the volume as a function of $\mathbf{L}$ is given by Mirzakhani's polynomial $V^{\mathrm{Mirz}}_{g,n}(\mathbf{L})$.  For larger angles, Mirzakhani's polynomial may no longer represent the volume.  Nevertheless, for any $\mathbf{L}\in i[0,2\pi)^n$ the volume $\mathrm{Vol}\big(\cM^{\text{hyp}}_{g,n}(\mathbf{L})\big)$ can be shown to be still given by a polynomial, and in the simplest case of $n=2$ its relation to Mirzakhani's polynomials is given in the following theorem.
\begin{restatable}{theorem}{thmone}
\label{g,2 volume}
    Let $g\in\bbZ_{>0}$ and $\theta_1,\theta_2\in(0,2\pi)$ be such that $\theta_1+\theta_2> 2\pi$. Then the volume of $\cM^{\text{hyp}}_{g,2}(i\theta_1,i\theta_2)$ is a polynomial in $\theta_1,\theta_2$ satisfying
    \begin{equation*}
        \mathrm{Vol}\big(\cM^{\text{hyp}}_{g,2}(i\theta_1,i\theta_2)\big) = V^{\mathrm{Mirz}}_{g,2}(i\theta_1,i\theta_2)+\int_0^\phi V^{\mathrm{Mirz}}_{g,1}(i\theta)\cdot\theta d\theta,
    \end{equation*}
    where $\phi=\theta_1+\theta_2-2\pi$.
\end{restatable}
Theorem~\ref{g,2 volume} is an example of a {\em wall-crossing} phenomenon.  The moduli spaces of hyperbolic surfaces $\cM^{\text{hyp}}_{g,2}(i\theta_1,i\theta_2)$ are parametrised by $(\theta_1,\theta_2)\in[0,2\pi)^2$ which is partitioned into two chambers defined %by $\theta_1+\theta_2< 2\pi$, respectively $\theta_1+\theta_2> 2\pi$, and a 
by the complement of the
wall $\theta_1+\theta_2=2\pi$.  Associated to each chamber is a polynomial, giving the volumes of the corresponding moduli spaces, and the difference between the polynomials is the wall-crossing polynomial, which arises in Theorem~\ref{g,2 volume} as an integral of a simpler volume polynomial.

Theorem~\ref{g,2 volume} is a special case of a more general wall-crossing formula which applies to any $g$ and $n$.  To state this, it will be convenient to replace the angles with linearly related parameters via the linear function
\[ a(\theta):=1-\frac{\theta}{2\pi}
\]
so that $a_j=a(\theta_j)\in(0,1]$ for all $j$.    Write $\mathbf{a}(\boldsymbol{\theta})=(a(\theta_1),...,a(\theta_n))$ for $\boldsymbol{\theta}=(\theta_1,...,\theta_n)\in[0,2\pi)^n$. This produces the following space of parameters for the moduli spaces $\modm_{g,n}^{\text{hyp}}(i\theta_1,...,i\theta_n)$:
\[ \cd_{g,n} :=\{\mathbf{a}\in(0,1]^n\mid\sum_1^n a_j>2-2g\}.\]
The condition $\sum a_j>2-2g$ is automatically satisfied when $g>0$.  It is necessary when $g=0$, to ensure the existence of a hyperbolic metric which requires an upper bound $\sum_{j=1}^n\theta_j<2\pi(n-2)=2\pi|\chi|$.   

For each $J\subset\mathbf{n}=\{1,...,n\}$ with $|J|>1$ define a {\em wall} $W_J\subset\cd_{g,n}$ by
\[W_J=\{\mathbf{a}\in\cd_{g,n}\mid\sum_{j\in J}a_j=1\}. \]
The components of $\cd_{g,n}\setminus\cup_JW_J$ are {\em chambers} of $\cd_{g,n}$.  A chamber $\cc$ is defined to be incident to a wall $W_J$ if $\overline{\cc}\cap W_J\subset\overline{\cc}$ is of codimension one. %contains points unique to $W_J$, i.e. contained in no other wall.  In particular, t
By crossing the wall, this determines a unique chamber $\cc'\neq\cc$ also incident to $W_J$ and a path from $\cc$ to $\cc'$ is a {\em simple} wall-crossing.  If $\sum_{j\in J}a_j>1$ on $\cc$, so that $\sum_{j\in J}a_j<1$ on $\cc'$, then we say the chamber $\cc$ is {\em above} the wall $W_J$, and write $\cc>\cc'$.   Simple wall-crossings induce a partial ordering on the collection of chambers of $\cd_{g,n}$.  There is a unique maximal chamber $\cc^\mathrm{M}\subset\cd_{g,n}$ which we call the {\em main} chamber.  For $g>0$ there is a unique minimal chamber $\cc^\mathrm{L}\subset\cd_{g,n}$, while for $g=0$ one can associate a minimal chamber $\cc_j\subset\cd_{0,n}$, defined in \eqref{ming=0}, for each $j\in\{1,...,n\}$.  Clearly $\cc^\mathrm{M}=\cc^\mathrm{L}$ when $n=1$. Theorem~\ref{g,2 volume} gives a wall-crossing formula across the wall $W_{\hspace{-.5mm}\{1,2\}}$ separating $\cd_{g,2}$ into $\cc^\mathrm{M}$ and $\cc^\mathrm{L}$.

It is shown in Section~\ref{sec:hass} that $\mathrm{Vol}\big(\cM^{\text{hyp}}_{g,n}(i\boldsymbol{\theta})\big)\in\br[\theta_1,...,\theta_n]$ in each chamber $\cc\subset\cd_{g,n}$.  Denote this volume by $V_{g,\cc}(i\boldsymbol{\theta})$, so for $\mathbf{a}(\boldsymbol{\theta})\in\cc$
\[ V_{g,\cc}(i\boldsymbol{\theta})=\mathrm{Vol}\big(\cM^{\text{hyp}}_{g,n}(i\boldsymbol{\theta})\big)\in\br[\theta_1,...,\theta_n].
\]
 The polynomial $V_{g,\cc}$ is defined via algebraic geometry in Definition~\ref{def:vol}. 
We have
\[ V_{g,\cc^\mathrm{M}}=V^{\mathrm{Mirz}}_{g,n}.
\]
It is this comparison with Mirzakhani's polynomials that explains the use of $i\boldsymbol{\theta}$ instead of the more natural $\boldsymbol{\theta}$.  The collection of polynomials $V_{g,\cc}(i\boldsymbol{\theta})$ forms a piecewise polynomial function on $\cd_{g,n}$ which varies continuously and differentiably across walls.   Given a chamber $\cc\subset\cd_{g,n}$ incident to and above a wall $W_S$, hence uniquely determining $\cc'$ incident to and below $W_S$, we define the wall-crossing polynomial $\wc_{\cc,S}\in\br[\theta_1,...,\theta_n]$ by
\[ 
V_{g,\cc'}(i\boldsymbol{\theta}) = V_{g,\cc}(i\boldsymbol{\theta})+\wc_{\cc,S}(\boldsymbol{\theta}).
\]
%(If $\cc_+$ lies above the wall $W_J$, i.e. points of $\cc_+$ satisfy $\sum_{j\in J}a_j>1$, so points of $\cc_-$ satisfy $\sum_{j\in J}a_j<1$ then $WC_{\cc_+,\cc_-}(i\boldsymbol{\theta})$ is positive on the chamber $\cc_+$.  REMOVE?  TOO DETAILED?)
 
For $S\subset\mathbf{n}$, $\boldsymbol{\theta}_S\in[0,2\pi)^S$ is the restriction of $\boldsymbol{\theta}\in[0,2\pi)^n$ and $S^c$ is the complement of $S$ in $\mathbf{n}$.   For any wall $W_S\subset\cd_{g,n}$, there is a natural quotient construction that sends any chamber $\cc\subset\cd_{g,n}$ to a quotient chamber $\cc/S\subset\cd_{g,n-|S|+1}$, defined precisely in \eqref{quotient}.

\begin{restatable}{theorem}{thmtwo}
\label{wall-crossing formula}
Let $\cc\subset\cd_{g,n}$ be incident to and above a wall $W_S$.  Then
\begin{equation} \label{WCF}
\wc_{\cc,S}(\boldsymbol{\theta})=\int_{0}^{\phi_S} V_{g,\cc/S}(i\boldsymbol{\theta}_{S^c},i\theta)\cdot V_{0,\cc_1}(i\theta,i\boldsymbol{\theta}_{S})  \cdot \theta \cdot d\theta
\end{equation}
where $\phi_S=\sum_{j\in S}\theta_j-2\pi(|S|-1)$ and $\cc_1$ is the minimal chamber corresponding to the entry $i\theta$.
\end{restatable}
The polynomials $V_{0,\cc_1}(i\theta,i\boldsymbol{\theta}_{S})$ are calculated explicitly in Proposition~\ref{lem:lightchamber} leading to the more explicit formula \eqref{wckernel}. %, or the cusp in the geometric picture.  
A useful geometric viewpoint of Theorem~\ref{wall-crossing formula}, which anticipates the proof, is that of a nodal degeneration of a hyperbolic surface where the cone angles $\boldsymbol{\theta}_{S}$ coalesce to form a genus 0 cuspidal surface that bubbles off. % Knowing these volumes means that \eqref{WCF} can be written more explicitly, which we do in Theorem~\ref{thm:wall2}. The formula given here is less geometrically suggestive.

 When $|S|=2$, $V_{0,\cc_1}(i\theta,i\boldsymbol{\theta}_{S})=1$ and if also $n=2$ then $V_{g,\cc/S}(i\boldsymbol{\theta}_{S^c},i\theta)=V_{g,\cc^\mathrm{M}}(i\theta)$ and Theorem~\ref{wall-crossing formula} specialises to Theorem~\ref{g,2 volume}.

Via relations between intersection numbers on $\overline{\modm}_{g,n}$, the polynomials $V^{\text{Mirz}}_{g,n}(\mathbf{L})$ were proven in \cite{DNoWei} to satisfy the following relations:
\begin{subequations} \label{limitdil}
\begin{gather}
V^{\text{Mirz}}_{g,n+1}({\bf L},2\pi i)=\sum_{k=1}^n\int_0^{L_k}L_kV^{\text{Mirz}}_{g,n}({\bf L})dL_k, \label{limitdil1}\\
\frac{\partial V^{\text{Mirz}}_{g,n+1}}{\partial L_{n+1}}({\bf L},2\pi i)=2\pi i(2g-2+n)V^{\text{Mirz}}_{g,n}({\bf L}).\label{limitdil2}
\end{gather}
\end{subequations}
%These relations were generalised in \cite{DuSim}. 
These results generalise to volume polynomials outside of the main chamber as follows.   We write $\theta_j\to 2\pi$, equivalently $a_j\to0$, to mean a path in $\cd_{g,n}$ along which walls are crossed only from above to below.  We say that a chamber is \emph{flat in the $i$th coordinate} if there is a path $\theta_i\to 2\pi$ along which only walls $W_S$ with $|S|=2$ are crossed, and \emph{light in the $i$th coordinate} if there is a path $\theta_i\to 2\pi$ along which no walls are crossed---see Definitions~\ref{defn:flat} and \ref{defn:light}.
Let $\cc\subset\cd_{g,n+1}$ be light in the $(n+1)$th coordinate. Then
\begin{equation}\label{eqn:incidentzero}
V_{g,\cc}(i\theta_1,\dots,i\theta_{n},2\pi i)=0.
\end{equation}
This is a consequence of a degeneration of the Weil-Petersson form in the $2\pi$ limit, and also is proven via an algebro-geometric argument which appears as Lemma~\ref{Thm:2pilimits}.
Let  $\cb\subset\cd_{g,n+1}$ be any chamber and $\{\cb_j,S_j\}_{j=1}^k$ be any series of simple wall-crossings from $\cb=\cb_1$ to a chamber $\cb_{k+1}$ that is light in the $(n+1)$th coordinate.  A consequence of Theorem~\ref{wall-crossing formula} and \eqref{eqn:incidentzero} is the following equality:
\begin{equation}\label{eqn:incidentzerowc}
V_{g,\cc}(i\theta_1,\dots,i\theta_n,2\pi i)=-\sum_{j=1}^k\wc_{\cb_j,S_j}(\theta_1,\dots,\theta_n,2\pi )
\end{equation}
obtained by applying \eqref{WCF} multiple times to \eqref{eqn:incidentzero}.  In particular, when applied to the main chamber $\cb=\cc^M$ this produces the relation \eqref{limitdil1}.  An understanding of  \eqref{limitdil1} via a $2\pi$ cone angle limit was unknown previously, and this proof via a degeneration of the volume form in the $2\pi$ cone angle limit combined with a sequence of wall-crossings gives such a viewpoint.
%$$V_{g,\cc}(i\theta_1,\dots,i\theta_{n+1})=V_{g,\cc_{k+1}}(i\theta_1,\dots,i\theta_{n+1})-\sum_{j=1}^k\wc_{\cc_j,S_j}(i\theta_1,\dots,i\theta_{n+1}).$$

Given a chamber $\cc\subset\cD_{g,n+1}$, there is a unique chamber $\cc|_{\b{n}}$ obtained by forgetting the $(n+1)$th point, defined precisely in \eqref{restriction}. %The notion of a chamber being \emph{flat in the $i$th coordinate} is provided in Definition~\ref{defn:flat}. This is a weaker condition than the chamber being light in the $i$th coordinate. 
\begin{restatable}{theorem}{thmfour}%\cite{ETu2Dd}
 \label{Thm:Dilaton}
Let $\cc\subset\cd_{g,n+1}$ be flat in the $(n+1)$th coordinate. Then
\begin{equation}\label{eqn:derivative}
\frac{\partial V_{g,\cc}}{\partial\theta_{n+1}}(i\theta_1,\dots,i\theta_n,2\pi i)=-2\pi \big(2g-2+|q(\cc)|+\sum_{j\notin q(\cc)} a(\theta_j)\big)V_{g,\cc|_{\b{n}}}(i\theta_1,\dots,i\theta_n)
\end{equation}
where $q(\cc)=\{j\mid W_{\{j,n+1\}}\text{ is crossed as }\theta_{n+1}\to2\pi\}$.
\end{restatable}
%Further, let  $\cb\subset\cd_{g,n+1}$ be any chamber and $\{\cb_j,S_j\}_{j=1}^k$ such that $n+1\in S_j$ and $|S_j|\geq 3$ be a series of simple wall-crossings from $\cb=\cb_1$ to the unique chamber $\cb_{k+1}$ that is flat in the $(n+1)$th coordinate. Then \begin{align}\label{eqn:derivativegeneral}\begin{split}\frac{\partial V_{g,\cb}}{\partial\theta_{n+1}}(i\theta_1,\dots,2\pi i)&=-2\pi  (2g-2+|q(\cb)|+\sum_{j\notin q(\cb)} a(\theta_j))V_{g,\cb|_{\b{n}}}(i\theta_1,\dots,i\theta_n)\\&-\sum_{j=1}^k\frac{\partial\wc_{\cb_j,S_j}}{\partial \theta_{n+1}}(i\theta_1,\dots,i\theta_{n},2\pi i)\end{split}\end{align}
For the main chamber $\cc^M$, $V_{g,\cc^M}=V^{\text{Mirz}}_{g,n+1}$ and $V_{g,\cc^M|_{\b{n}}}=V^{\text{Mirz}}_{g,n}$, and $|q(\cc^M)|=n$, so that \eqref{eqn:derivative} coincides with the second equation of \eqref{limitdil}.  When $\cc$ is light in the $(n+1)$th coordinate, so that $q(\cc^M)=\emptyset$, then the factor $-2\pi (2g-2+\sum a(\theta_j))$ equals the negative of the area of each fibre in the universal curve.
%Results in \cite{ANoVol} described occurrences when the relations \eqref{limitdil} do and do not have geometric meaning. Here, we do the same with Theorem~\ref{Thm:Dilaton}. 
Equation \eqref{eqn:derivative} is a relation between the volume polynomials in any chamber flat in the $(n+1)$th coordinate.  If instead we consider only chambers light in the $(n+1)$th coordinate then  \eqref{eqn:derivative} essentially gives the derivative of the piecewise polynomials $\mathrm{Vol}\big(\cM^{\text{hyp}}_{g,n}(i\boldsymbol{\theta})\big)$ which was proven previously in \cite[(1.3b)]{ETu2Dd}. In Corollary~\ref{cor:generalchamber} we use wall-crossing polynomials to obtain a generalisation of \eqref{eqn:derivative} that holds for certain chambers that are not necessarily flat in the $(n+1)$th coordinate.

 The proofs of Theorems~~\ref{wall-crossing formula} and \ref{Thm:Dilaton} allow the existence of boundary geodesics, provided the lengths of the added boundary components are sufficiently small (see \cite[Lemma 2.3]{ANoVol}).  Each boundary geodesic is assigned the weight of a cusp in order to determine the appropriate chamber.  The same convention associates Mirzakhani's polynomial for geodesic boundary lengths to the main chamber.  

%Theorem~\ref{Thm:Dilaton} is proven in \cite{ETu2Dd} for moduli spaces with cone point boundary components together with boundary geodesics.  The proof in this paper also applies to this case provided the lengths of the added boundary component are sufficiently small (see \cite[Lemma 2.3]{ANoVol}).  Each boundary geodesic is assigned the weight of a cusp in order to determine the appropriate chamber.  The same convention associates Mirzakhani's polynomial for geodesic boundary lengths to the main chamber.  Similarly, Theorem~\ref{wall-crossing formula} allows boundary geodesics of sufficiently small length.

%All three of Theorems~\ref{g,2 volume},~\ref{wall-crossing formula} and~\ref{Thm:Dilaton} allow only for moduli spaces with purely cone point and cuspidal boundary components. This is done to maximise readability of formulas by sticking to the main ideas of this paper. In fact, the main theorems mostly still hold when any number of boundary geodesics, taking on $0$-weight, are added into the formulas. The only condition needed is that the lengths of each boundary component be sufficiently small, in line with \cite[Lemma 2.3]{ANoVol}.
  
The space of parameters $\cd_{g,n}$ coincides with Hassett's space of stability conditions \cite{HasMod} on the moduli space of pointed curves $\modm_{g,n}$.   Hassett considers rational parameters in order to work over any field of characteristic zero, but it easily extends to real parameters.  Hassett uses the stability conditions to define a condition on nodal curves known as $\mathbf{a}$-stability that defines $\mathbf{a}$-stable curves, which coincides with usual stable curves  for $\mathbf{a}=1^n$.  The set of $\mathbf{a}$-stable curves leads to a compactification  $\overline{\modm}_{g,\mathbf{a}}$ depending only on the chamber $\cc$ containing $\mathbf{a}$, and generalising the Deligne-Mumford compactification.   %The relationship of Hassett's space of stability conditions with hyperbolic surfaces is due to 
The association of moduli spaces of cone hyperbolic surfaces to Hassett's space of stability conditions arises by considering a natural compactification $\overline{\modm}_{g,n}^{\text{hyp}}(i\boldsymbol{\theta})$ obtained by including nodal hyperbolic surfaces with cusps at nodes. For $\mathbf{a}=\mathbf{a}(\boldsymbol{\theta})$, the compactification is proven in \cite{ANoVol} to be isomorphic to Hassett's compactification 
\begin{equation} \label{Hasiso} \overline{\modm}_{g,n}^{\text{hyp}}(i\boldsymbol{\theta})\cong\overline{\modm}_{g,\mathbf{a}}.\end{equation}
By relating volumes of the smooth moduli space $\modm_{g,n}^{\text{hyp}}(i\boldsymbol{\theta})$ to cohomology classes on the natural compactification $\overline{\modm}_{g,n}^{\text{hyp}}(i\boldsymbol{\theta})$, we can calculate volumes via intersection theory on Hassett's compactifications $\overline{\modm}_{g,\mathbf{a}}$.

In Section~\ref{sec:hass} we write the volumes of moduli spaces of conical hyperbolic surfaces in terms of intersection numbers over Hassett's compactifications and in particular show the piecewise polynomial behaviour of the volumes.  The proofs of Theorems~\ref{g,2 volume},~\ref{wall-crossing formula} and~\ref{Thm:Dilaton} are given in Section~\ref{main}.  In Section~\ref{sec:ex} we calculate volumes in genus $0$ for particular chambers, in particular we give explicit formulae for the volume polynomials $V_{\cc_1}$ appearing in \eqref{WCF}, and also volumes of Losev-Manin Spaces as well as the volumes of all chambers $\mathrm{Vol}\big(\cM^{\text{hyp}}_{g,n}(i\boldsymbol{\theta})\big)$ for $(g,n)=(0,4),(0,5)$ and $(1,2)$.  In many of the computations included in this section, the three main theorems are necessary, and this section serves to highlight how they can be used in practice.

%The methods of proof used in this paper are algebro-geometric whereas the statements are differential geometric. 

\textit{Acknowledgements.} Research of LA is supported by an Australian Government Research Training Program Scholarship and an Elizabeth \& Vernon Puzey Scholarship. Research of SM is supported by the ERC Advanced Grant ``SYZYGY'' and the DECRA Grant DE220100918 from the Australian Research Council.  Research of PN is supported under the Australian Research Council Discovery Projects funding scheme project number DP180103891.

%%%%%%%%%%%%%%%%%%%%%%%%%%%%%%%%%%%%%
%%%%%%%%%%%%%%%%%%%%%%%%%%%%%%%%%%%%%
\section{Compactifications and intersection theory}   \label{sec:hass}
%Intro to the section

There is a natural compactification of $\cM^{\text{hyp}}_{g,n}(i\boldsymbol{\theta})$ by nodal hyperbolic surfaces with cusps at nodes.  Via algebro-geometric methods, Hassett \cite{HasMod} produced compactifications $\overline{\modm}_{g,\mathbf{a}}$ of $\cM_{g,n}$ using weighted stable curves depending on weights $\mathbf{a}=(a_1,...,a_n)\in(0,1]^n$.  For $a=a(\theta_i)$, the compactifications by nodal hyperbolic surfaces and those by weighted stable curves turn out to be homeomorphic.  This allows one to study volumes of $\cM^{\text{hyp}}_{g,n}(i\boldsymbol{\theta})$ via intersection theory on $\overline{\modm}_{g,\mathbf{a}}$.
%%%%%%%%%%%%%%%%%%%%%%%%%%%%%%%%%%%%%%

\subsection{Hassett's space of stability conditions}\label{sec:Hass}
The space of parameters 
\[ \cd_{g,n} :=\{\bold{a}\in(0,1]^n\mid\sum a_i>2-2g\}\]
which represents cone angles via $a_i=a(\theta_i)$
is used by Hassett \cite{HasMod} to define stability conditions which governs a collection of permitted nodal curves used to produce alternative compactifications of $\cM_{g,n}$ similar to the Deligne-Mumford compactification.  On any smooth pointed curve $C$, the parameters define a weighted Euler characteristic
\[\chi_{g,\bold{a}}(C)= 2-2g(C)-\sum_{j=1}^na_j.\]
The following condition on any nodal curve, known as $\mathbf{a}$-stability, coincides with the  Deligne-Mumford condition of stabilty for $\mathbf{a}=1^n$.
%These compactifications consist of nodal curves with labeled points weighted by $\mathbf{a}=(a_1,...,a_n)$ and each irreducible component must be stable with respect to these weights.   
\begin{definition}\cite{HasMod}   \label{astable}
For $\mathbf{a}\in\cd_{g,n}$, a pointed nodal curve $(C,p_1,...,p_n)$  is $\mathbf{a}$-stable when the points $p_1,...,p_n$ lie in the smooth locus of $C$,  
$\chi_{g,\bold{a}}(C')<0$ for each irreducible component $C'\to C$, %$(C',\{p_j\mid j\in J\},N)\to (C,p_1,...,p_n)$
  %$N\subset C$ % avoid the nodal points $N\subset C$
and if $\{p_j=p\mid j\in J\}$ collide then $\displaystyle\sum_{j\in J}a_j\leq 1$.  
\end{definition}
Note that we define an irreducible component $C'\to C$ to mean a smooth connected component of the normalisation of $C$, the weights $\mathbf{a}=(a_1,...,a_n)$ apply only to the points $p_j\in C'$, and nodes are given weight 1.
%The restriction $\cd_{g,n}\cap\bq^n$ is used by Hassett in order to work over any field of characteristic zero, but his methods easily extend to all of $\cd_{g,n}$.
\begin{thrm}\cite{HasMod}
The set $\overline{\modm}_{g,\mathbf{a}}$ of all  $\mathbf{a}$-stable curves defines a compactification  of $\cM_{g,n}$.
\end{thrm}
A %consequence of Definition~\ref{astable} is that a 
weight $1$ point gives a usual labeled point in the sense of Deligne-Mumford stability.   Points of weight 1 cannot collide with each other on nodal curves, since a set of weighted points $\{p_j\mid j\in J\subset\{1,...,n\}\}$ for $|J|>1$ can collide only when $\sum_{j\in J}a_j<1$, hence $\overline{\modm}_{g,1^n}$ is isomorphic to the Deligne-Mumford compactification.
\begin{definition}   \label{Dd}
For any $S\subset\{1,\dots,n\}$ and $\bold{a}\in\cd_{g,n}$, define  the locus of curves with a genus 0 reducible component containing precisely the points $\{p_j\mid j\in S\}$  by
\begin{equation}\label{d0S}
\delta_{0:S}\subset\Mbar{g}{\bold{a}},\qquad \text{if }\sum_{j\in S}a_j>1;\end{equation}
and the locus of curves where the points $\{p_j\mid j\in S\}$ collide by
\begin{equation}\label{D_S}
D_S\subset\Mbar{g}{\bold{a}},\qquad \text{if }\sum_{j\in S}a_j\leq1.\end{equation}

\end{definition}
Note that each $\delta_{0:S}$ defines a divisor in $\Mbar{g}{\bold{a}}$ whereas $\text{codim }D_S=|S|-1$ so only the $D_{i,j}$ define divisors.
%This generalises the limit proven in \cite{DNoWei} via relations between intersection numbers on $\overline{\modm}_{g,n}$, that Mirzakhani's polynomials $V_{g,n}(\mathbf{L})$ satisfy $\lim_{\theta\to 2\pi i}V_{g,n}(i\theta,0,...,0)=0.$

%%%%%%%%%%%%%%%%%%%%%%%%%%%%%%%%%%%

%%%%%%%%%%%%%%%%%%%%%%%%%%%%%%%%%%%
\subsection{Chambers}
For any $J\subset\{1,...,n\}$ such that $|J|\geq 2$ define the wall
\[W_J=\{\mathbf{a}\in\cd_{g,n}\mid\sum_{j\in J}a_j=1\}.\]
The components of $\cd_{g,n}\setminus\cup_JW_J$ are {\em chambers} of $\cd_{g,n}$.  Given any $\mathbf{a}\in\cd_{g,n}\setminus\cup_JW_J$, 
for any $\mathbf{b}\in\cd_{g,n}$ that lies in the same chamber, a nodal curve $C$ is $\mathbf{a}$-stable if and only if it is $\mathbf{b}$-stable.  This is because otherwise along a path from $\mathbf{a}$ to $\mathbf{b}$, there is a point $\mathbf{a}'$ where the stability condition in Definition~\ref{astable} changes, hence %either $\sum_{j\in J}a'_j= 1$ which changes the behaviour $\{p_j=p\mid j\in J\}$ for $\mathbf{a}$ or $\mathbf{b}$, or 
$\chi_{g,\bold{a}'}(C')=0$ %hence $g(C')=0$, there is exactly one nodal point and 
which implies that $\sum_Ja'_j=1$ for $J={\{j:p_j\in C'\}}$. %for all points $p_j$ the condition $\chi_{g,\bold{a}}(C')<0$ for any irreducible component $C'\to C$ can only change to $\chi_{g,\bold{b}}(C')>0$ if 
Thus, if $\mathbf{a}$ and $\mathbf{b}$ lie in the same chamber of $\cd_{g,n}$ then
\[\overline{\modm}_{g,\mathbf{a}}\cong \overline{\modm}_{g,\mathbf{b}}
\]
and for any chamber $\cc\subset\cd_{g,n}$ we define 
\[ \overline{\modm}_{g,\cc}:=\overline{\modm}_{g,\mathbf{a}},\quad\text{for any }\mathbf{a}\in\cc.
\]
Note that Hassett defines a coarser chamber decomposition on which the isomorphism type of $\overline{\modm}_{g,\mathbf{a}}$ remains the same. 
 
A chamber $\cc\subset\cd_{g,n}$ is uniquely determined by an order-preserving function on the power set $\cp(\mathbf{n})$,
\[\cc:\cp(\mathbf{n})\to\{0,1\}\text{ defined on }J\subset\mathbf{n}\text{ by:}\quad \cc(J)=0\Leftrightarrow\forall\mathbf{a}\in\cc,\ \sum_{j\in J}a_j\leq1\]
(which is also denoted $\cc$ by a slight abuse of notation).  It is order-preserving with respect to the natural partial order on $\cp(\mathbf{n})$ and the total order $0<1$ on $\{0,1\}$.
The maximal or main chamber $\cc^\mathrm{M}\subset\cd_{g,n}$ is defined by the function $\cc^\mathrm{M}(J)=1$ for any $|J|>1$, which corresponds to $\sum_{j=1}^n\theta_j<2\pi$.  The minimal or light chamber $\cc^\mathrm{L}\subset\cd_{g,n}$ is defined by the function $\cc^\mathrm{L}(J)=0$ for all $J$.  When $g=0$, $\cc^\mathrm{L}=\emptyset$ so we instead define 
\begin{equation}  \label{ming=0}
\cc_j\subset\cd_{0,n},\qquad\cc _j(J)=1\quad\Leftrightarrow\quad\{j\}\subsetneq J\text{ or }\{j\}^c=J.
\end{equation} 

A chamber $\cc\subset\cd_{g,n}$ incident to and above a wall $W_S$, so in particular $\cc(S)=1$, determines a unique chamber $\cc'$ that satisfies $\cc'(S)=0$ and $\cc'(J)=\cc(J)$ for all $J\neq S$.  Then $\cc'$ is incident to and below $W_S$, and $\cc$ is related to $\cc'$ by a simple wall-crossing.

There are two important constructions concerning chambers that we will need in this work.

For any $S\subset \mathbf{n}$, there is a natural quotient construction that sends any chamber $\cc:P(\mathbf{n})\to\{0,1\}$ to a quotient chamber $\cc/S:P(\mathbf{n}/S)\to\{0,1\}$ defined by 
\begin{equation}   \label{quotient}
(\cc/S)(\bar{J})=\left\{\begin{array}{cl}0,&J\subset S\\
\cc(J),& J\cap S=\emptyset\\
1,&\text{otherwise}
\end{array}\right.
\end{equation}
where $\bar{J}$ is the image of $J$ under the natural map $\mathbf{n}\to\mathbf{n}/S$.  

For any $T\subset \mathbf{n}$, there is a natural restriction construction that sends any chamber $\cc:P(\mathbf{n})\to\{0,1\}$ to a chamber $\cc|_T:P(T)\to\{0,1\}$ in $\cD_{g,|T|}$ defined by \begin{equation}   \label{restriction}
\cc|_T(J)= \cc(J)
\end{equation}
and uniquely defined by the commuting diagram:
\begin{center}\begin{tikzcd}[row sep=3em, column sep=3em,
    text height=1.5ex, text depth=0.5ex]
\cP(\b{n}) \arrow[r,"\cc"]   & \{0,1\} 
\\
  \cP(T)  \arrow[ru,swap,"\cc|_T"]  \arrow[u,"\cup"] & 
\end{tikzcd}
\end{center}
In particular, for a chamber $\cc\subset\cD_{g,n+1}$, the unique chamber obtained by forgetting the $(n+1)$th point is $\cc|_{\b{n}}$. We will need this construction when dealing with forgetful morphisms.

%%%%%%%%%%%%%%%%%%%%%%%%%%%%%%%%%%%

%%%%%%%%%%%%%%%%%%%%%%%%%%%%%%%%%%%
\subsection{Reduction morphisms}\label{sec:HassRed}
For any two sets of weight data $\bold{a}$ and $\bold{b}$ such that $a_j\geq b_j$ for all $j$ we write $\bold{a}\geq\bold{b}$ and there exists a birational reduction morphism~\cite[\S 4]{HasMod}
$$\pi_{\bold{a},\bold{b}}:\Mbar{g}{\bold{a}}\longrightarrow \Mbar{g}{\bold{b}}$$
that contracts the rational tails divisors $\delta_{0:S}$, defined in \eqref{D_S}, for $|S|\geq 3$ where $\sum_{j\in S}b_j\leq 1$ but $\sum_{j\in S}a_j> 1$. We define the \emph{reduction sets} of such a reduction morphism as
$$ \mathcal{S}_{\bold{a},\bold{b}}:=  \{S\subset\{1,\dots,n\}| \sum_{j\in S}b_j\leq 1\text{ and  }\sum_{j\in S}a_j>1\}$$
and we refer to $\pi_{\bold{a},\bold{b}}$ as a \emph{simple reduction morphism} if $|\mathcal{S}|=1$.
Further, we have~\cite[Corollary 5.4]{AlexeevGuy} 
\begin{equation} \pi_{\bold{a},\bold{b}}^*(\psi_j)=\psi_j-\sum_{S\in N(j)} \delta_{0:S} \label{redmor} \end{equation}
where 
$$N(j):=\{S\in  \mathcal{S}_{\bold{a},\bold{b}} \mid j\in S\}.$$
For $\{i,j\}\in\mathcal{S}_{\bold{a},\bold{b}}$ we have ${\pi_{\bold{a},\bold{b}}}_*(\delta_{0:\{i,j\}})=D_{i,j}$, defined in \eqref{D_S}, and
$$\pi_{\bold{a},\bold{b}}^*D_{i,j}=\begin{cases}\delta_{0:\{i,j\}}+\sum_{S\in N(i,j)}\delta_{0:S} &\text{if $a_i+a_j>1$ }\\ 
D_{i,j}+\sum_{S\in N(i,j)}\delta_{0:S} &\text{if $a_i+a_j\leq1$ }\\
\end{cases}$$
where
$$N(i,j):=\{S\in  \mathcal{S}_{\bold{a},\bold{b}} \mid i,j\in S \}    $$
Clearly for $\bold{a}\geq\bold{b}\geq \bold{c}$ the relation 
$$\pi_{\bold{a},\bold{c}}=\pi_{\bold{b},\bold{c}}\circ \pi_{\bold{a},\bold{b}}$$
holds.

%%%%%%%%%%%%%%%%%%%%%%%%%%%%%%%%%%%
\subsection{Forgetful morphisms}\label{sec:HassForget}
For any weight vector $\bold{a}=(a_1,\dots,a_{n})$ such that the subvector $\bold{a}'=(a_{i_1},\dots,a_{i_r})$ obtained by forgetting entries is also a weight vector, there exists a forgetful morphism~\cite[\S 4]{HasMod}
$$\varphi_{\bold{a},\bold{a}'}:\Mbar{g}{\bold{a}}\longrightarrow \Mbar{g}{\bold{a}'}.$$
where the image of an $\bold{a}$-stable curve is obtained by stable reduction, that is, by successively contracting the rational components that do not remain $\bold{a}'$-stable after forgetting the specified marked points.

We specialise to the case that $r=n-1$, that is, only one point is forgotten. Without loss of generality let $\bold{a}=(a_1,\dots,a_{n})$ and let the last point be forgotten. We define the \emph{contraction set} as
$$\mathcal{T}_{\bold{a},\bold{a}'}:=\{S\subset \{1,\dots,n\}\mid n\in S\text{ and }1-a_n<\sum_{j\in S-\{n\}}a_j\leq1\}.$$
This set specifies exactly the rational tails that will become unstable after forgetting the $n$-th point. We have~\cite[Corollary 5.8]{AlexeevGuy} 
\begin{equation} 
\varphi_{\bold{a},\bold{a}'}^*(\psi_j)=\psi_j-\sum_{S\in M(j)} \delta_{0:S} \label{formor} \end{equation}
where 
$$M(j):=\{S\in  \mathcal{T}_{\bold{a},\bold{a}'} \mid j\in S\}.$$
Further, for $\{i,j\}\in \mathcal{T}_{\bold{a},\bold{a}'}$ we have
$$ \varphi_{\bold{a},\bold{a}'}^*D_{i,j}=D_{i,j}+\sum_{S\in M(i,j)}\delta_{0:S}   $$
for 
$$M(i,j):=\{S\in  \mathcal{T}_{\bold{a},\bold{a}'} \mid i,j\in S\}.$$ 
While for any $T\subset \{1,\dots,n-1\}$ such that $T\cup\{n\}\not\in \mathcal{T}_{\bold{a},\bold{a}'}$ we have
$$\varphi_{\bold{a},\bold{a}'}^*\delta_{0:T}=\delta_{0:T}+\delta_{0:T\cup\{n\}}.$$

\begin{lemma}\label{lem:forgetflat}
The forgetful morphism 
$$\varphi_{\bold{a},\bold{a}'}:\Mbar{g}{\bold{a}}\longrightarrow \Mbar{g}{\bold{a}'}$$
for $\bold{a}=(a_1,\dots,a_{n})$ and $\bold{a}'=(a_{1},\dots,a_{n-1})$ is flat if and only if the contraction set $\mathcal{T}_{\bold{a},\bold{a}'}$ contains only sets of order at most $2$.
\end{lemma}

\begin{proof}
In these cases $\Mbar{g}{\bold{a}}$ is isomorphic to the universal curve over $\Mbar{g}{\bold{a}'}$. In all other cases there exists some $S\in\mathcal{T}_{\bold{a},\bold{a}'}$ with $|S|\geq 3$ which gives $\varphi_{\bold{a},\bold{a}'}$ fibres of dimension at least $2$ over the image of $\delta_{0:S}$ providing a contradiction to flatness.
\end{proof}
This Lemma motivates the following two definitions.
\begin{definition}\label{defn:flat}
We refer to a signature $\bold{a}$ as \emph{flat in the $i$th coordinate} or that the $i$th marked point is a \emph{flat point} if the contraction set contains only sets of order at most $2$ for the morphism forgetting the $i$th point. Hence the forgetful morphism forgetting the $i$th point provides the universal curve over the target. We refer to a chamber $\cc\subset \cd_{g,n}$ as so if it contains a signature $\bold{a}\in\cc$ that is flat in the $i$th coordinate.
\end{definition}
 A stronger condition than the forgetful morphism being flat is when the contraction set is empty. 
\begin{definition}\label{defn:light}
We refer to a signature $\bold{a}$ as \emph{light in the $i$th coordinate} or that the $i$th marked point is a \emph{light point} if the contraction set is empty for the morphism forgetting the $i$th point. We refer to a chamber $\cc\subset \cd_{g,n}$ as so if it contains a signature $\bold{a}\in\cc$ that is light in the $i$th coordinate.
\end{definition}

\subsection{Cohomology classes}
For $\bold{a}=(a_1,\dots,a_n)$ with $0<a_i\leq1$ and $2g-2+\sum a_i>0$, the classes  $\psi_j\in H^2(\Mbar{g}{\bold{a}},\bq)$ are defined  by 
$$\psi_j:=c_1(s_j^*(\omega_\pi))$$ 
where $\omega_\pi$ is the relative dualising sheaf of the universal family
$$\pi:\mathcal{C}_{g,\bold{a}}\longrightarrow \Mbar{g}{\bold{a}}$$
and $s_j$ is the section corresponding to the $j$th point.

The following classes in $H^2(\Mbar{g}{\bold{a}},\bq)$ generate the rational Picard group: $\lambda$, the first Chern class of the Hodge bundle, $\psi_j$ for $j=1,\dots,n$, the classes $D_{i,j}$ for $a_i+a_j\leq1$ defined in \eqref{D_S}. %, and the classes of the irreducible components of the boundary. 
For any $S\subset\{1,\dots,n\}$, the classes $\delta_{0:S}$ with $\sum_{j\in S}a_j> 1$ defined in \eqref{d0S}, and $\delta_{i:S}$ for $1\leq i\leq g/2$, given by the class of the locus of curves with a separating node that separates the curve such that one component has genus $i$ and contains precisely the markings of $S$.  %Note that $\delta_{i:S}=\delta_{g-i:S^c}$   
Also, $\delta_{\rm{irr}}$ the class of the locus of curves with a non-separating node. %Hence $\delta_{i:S}=\delta_{g-i:S^c}$ and we require $\sum_{j\in S}a_j> 1$ for $i=0$ and $\sum_{j\in S^c}a_j> 1$ for $i=g$. 
For $g\geq 3$ these divisors freely generate the rational Picard group.

%%%%%%%%%%%%%%%%%%%%%%%%%%%%%%%%%%%
\subsubsection{The kappa class}\label{sec:kappa}
We define $\kappa_1:=\pi_*(c_1(\omega_\pi(\sum D_i))^2)$.
%the push forward of the square of the first Chern class of the relative dualising sheaf on the universal family. 
Note that this definition follows that of~\cite{ACoCom} and differs from the Miller-Morita-Mumford class $\pi_*(c_1(\omega_\pi)^2)$ though both classes appear in the literature with the same notation.

\begin{proposition}\label{prop:kappaclass}
The kappa class of $\Mbar{g}{\bold{a}}$ is given in terms of standard generators as
$$\kappa_1=12\lambda+\sum_{i=1}^n\psi_i-\delta+2\sum_{a_i+a_j\leq1}D_{i,j}.     $$
where $\delta$ refers to the generically nodal boundary classes, that is, it does not include the classes $D_{i,j}$.
\end{proposition}

\begin{proof}
If the signatures $\bold{a}$ and $\bold{b}$ give a simple reduction morphism $\alpha$ for the set $S\subset \{1,\dots,n\}$. We have the following diagram of the universal families 
 \begin{center}
\begin{tikzpicture}[description/.style={fill=white, inner sep=2pt}]
    \matrix (m) [matrix of math nodes, row sep=3em, column sep=3em,
    text height=1.5ex, text depth=0.5ex]
    { \mathcal{C}_{g,\bold{a}}& \mathcal{C}_{g,\bold{b}}\\
    \Mbar{g}{\bold{a}}& \Mbar{g}{\bold{b}}\\};
    \path[->, font=\scriptsize]
    (m-1-1) edge node[auto] {$p$} (m-1-2);
    \path[->, font=\scriptsize]
    (m-2-1) edge node[auto] {$\alpha$} (m-2-2);
    \path[->, font=\scriptsize]
    (m-1-1) edge node[auto] {$\pi$} (m-2-1);
      \path[->, font=\scriptsize]
    (m-1-2) edge node[auto] {$\zeta$} (m-2-2);
  \end{tikzpicture}
  \end{center}
with
\begin{eqnarray*}
p^*\omega_{\zeta}&=&\omega_\pi-E_S\\
p^*D_j&=&\begin{cases}D_j&\text{ for $j\not\in S$}\\ D_j+E_S&\text{ for $j\in S$}  \end{cases}
\end{eqnarray*}
where $E_S$ is the irreducible component of $\pi^*\delta_{0:S}$ that is contracted under $p$. Then using the notation $\kappa_1^{\bold{a}}$ and  $\kappa_1^{\bold{b}}$ for the kappa classes for $\Mbar{g}{\bold{a}}$ and $\Mbar{g}{\bold{b}}$ respectively, we obtain
\begin{eqnarray*}
\kappa_1^{\bold{b}}&=&\zeta_*([\omega_\zeta(\sum D_j)]^2)\\
&=& \alpha_*\pi_*((\omega_\pi+\sum D_j).p^*(\omega_\zeta+\sum D_j))\\
&=& \alpha_*\pi_*((\omega_\pi+\sum D_j).(\omega_\pi+\sum D_j-E_S+|S|E_S))\\
&=&\alpha_*\pi_*((\omega_\pi+\sum D_j)^2+(|S|-1)E_S.(\omega_\pi+\sum D_j))\\
&=&\alpha_*\pi_*((\omega_\pi+\sum D_j)^2+(|S|-1)E_S.\omega_\pi+(|S|-1)E_S.(\sum D_j))\\
&=&\alpha_*(\kappa_1^{\bold{a}}+(|S|-1)^2\delta_{0:S})\\
&=&\begin{cases}\alpha_*(\kappa_1^{\bold{a}})+D_{i,j}&\text{ for $S=\{i,j\}$}\\
\alpha_*(\kappa_1^{\bold{a}})&\text{for $|S|\geq 3$,}
\end{cases}
\end{eqnarray*}
which gives the pushforward formula
$$ \alpha_*(\kappa_1^{\bold{a}}) =\begin{cases}  \kappa_1^{\bold{b}}-D_{i,j}&\text{ for $S=\{i,j\}$}\\
\kappa_1^{\bold{b}}&\text{for $|S|\geq 3$.}\end{cases}
$$
The final observation is that for $j\in S$ we obtain
$$\alpha_*(\psi_j)=\alpha_*((\alpha^*\psi_j)+\delta_{0:S})=\begin{cases} \psi_j+D_{i,j} &\text{ for $S=\{i,j\}$}\\ 
\psi_j&\text{for $|S|\geq 3$,}
\end{cases}  $$
while $\alpha_*\psi_j=\psi_j$ for $j\not\in S$. The general formula for the class of $\kappa_1$ follows from these pushforward formulas and the known formula 
$$\kappa_1=12\lambda -\delta+\sum_{j=1}^n\psi_j$$ 
on the Deligne-Mumford compactification $\Mbar{g}{n}$.
\end{proof}

\begin{lemma}\label{lem:kappapullback1}
If the signatures $\bold{a}$ and $\bold{b}$ give a reduction morphism. Then $\pi_{\bold{a},\bold{b}}$ transforms the kappa class as follows:
$$(\pi_{\bold{a},\bold{b}})_*(\kappa_1^{\bold{a}})=\kappa_1^{\bold{b}}-\sum_{\{i,j\}\in\mathcal{S}_{\bold{a},\bold{b}}}D_{i,j}$$
and
$$\pi_{\bold{a},\bold{b}}^*(\kappa_1^{\bold{b}})=\kappa_1^{\bold{a}}+\sum_{S\in\mathcal{S}_{\bold{a},\bold{b}}}(|S|-1)^2\delta_{0:S}.$$
\end{lemma}

\begin{proof}
The first relation is proven in the proof of Proposition~\ref{prop:kappaclass}. The second relation follows from relating the class of $\kappa_1$ given in Proposition~\ref{prop:kappaclass} via the pullback formula for reduction morphisms given in \S\ref{sec:HassRed}.
\end{proof}

\begin{lemma}\label{lem:kappapullback2}
Consider the forgetful morphism $\varphi=\varphi_{\bold{a},\bold{a}'}$ for $\bold{a}=(a_1,\dots,a_n)$ and $\bold{a}'=(a_1,\dots,a_{n-1})$ with contraction set $\mathcal{T}= \mathcal{T}_{\bold{a},\bold{a}'} $ then 
$$\varphi^*\kappa_1^{\bold{a}'}= \kappa_1^{\bold{a}} -\psi_n-2\sum_{a_i+a_n\leq1}D_{i,n}+\sum_{S\in \mathcal{T}}(|S|-2)^2\delta_{0:S}.  $$
\end{lemma}

\begin{proof} The result follows from the observations that 
$$\varphi^*\delta=\delta-\sum_{S\in\mathcal{T}}\delta_{0:S}$$
and 
$$\sum_{\tiny{\begin{array}{ccc}&a_i+a_j\leq1,&\\&1\leq i,j\leq n-1&\end{array}}}\hspace{-10mm}\varphi^*D_{i,j}=\hspace{-5mm}\sum_{\tiny{\begin{array}{ccc}&a_i+a_j\leq1,&\\&1\leq i,j\leq n-1&\end{array}}}\hspace{-10mm}D_{i,j}+\sum_{S\in\mathcal{T}}{\tiny \frac{(|S|-1)(|S|-2)}{2}}\delta_{0:S}.$$
\end{proof}

In the rest of this paper, in an effort to manage notation, we will omit the indices on $\kappa_1$ as it is clear from context to which class we are referring.

\subsection{Limits of conical hyperbolic surfaces and compactifications}
There are natural compactifications $\overline{\modm}_{g,n}^{\text{hyp}}(\mathbf{L})$ formed via limits of families of conical hyperbolic surfaces.  They are isomorphic to Hassett's compactifications for corresponding stability conditions.  Such limits naturally produce nodal hyperbolic surfaces and also allow coalescing cone angles for large enough cone angles.  A brief description of the proof \cite{ANoVol} that this produces spaces isomorphic to Hassett's compactifications is as follows.  If $J\subset\{1,...,n\}$ defines a collection of cone points with angles satisfying $\sum_{j\in J}\theta_j<2\pi(|J|-1)$, the points repel, in the sense that as such a collection of points coalesce, a geodesic surrounding these points becomes shorter and produces a node before the cone points can meet.
% does not require $\sum_{j=1}^n\theta_j<2\pi$, rather it applies more generally to the local condition   and $\sum_{j\in J}\theta_j$ is not an integer multiple of $2\pi$, and produces a nodal curve in the limit as the $|J|$ points approach each other.  
The agreement with Hassett's local condition for the same behaviour is shown by:
\[\sum_{j\in J}\theta_j<2\pi(|J|-1)\ \Leftrightarrow\ \sum_{j\in J}(1-a_j)<|J|-1\ \Leftrightarrow\ \sum_{j\in J}a_j>1.\]
The limiting nodal curve has hyperbolic components, and each irreducible component minus its nodal points satisfies a negative weighted Euler characteristic condition. Hassett's stability condition on nodal curves given in Definition~\ref{astable}, produces the same negative weighted Euler characteristic condition on irreducible components, which guarantees existence of a conical hyperbolic metric in its conformal class, proven by McOwen in \cite{McOPoi}.  Hence the hyperbolic nodal curves and $\b{a}$-stable curves correspond.  On the other hand, if $\sum_{j\in J}\theta_j>2\pi(|J|-1)$, corresponding to $\sum_{j\in J}a_j<1$, in \cite{ANoVol} it is proven that cone angles can coalesce to produce a new angle $\theta=\sum_{j\in J}\theta_j+2\pi(1-|J|)$ corresponding to a new weight $a=\sum_{j\in J}a_j\leq 1$, again mirroring Hassett's stability condition. 

Schumacher and Trapani \cite{STrWei} generalised a construction by Wolpert \cite{WolChe},  of the Weil-Petersson form via the push forward of the square of curvature of a natural Hermitian line bundle.  This led them to define a symplectic form $\omega^{\text WP}_{\mathbf{L}}$ on $\modm_{g,n}\cong\cM^{\text{hyp}}_{g,n}(\boldsymbol{L})$ 
for a fixed choice of $\boldsymbol{L}=i\boldsymbol{\theta}\in i(0,2\pi]^n$ via a naturally defined Hermitian metric on the relative log-canonical line bundle over the universal curve.  Their construction uses the existence of an incomplete hyperbolic metric with cone angles $\theta_j$ on any fibre which varies smoothly in the family $S$ proven by McOwen \cite{McOPoi}.  The hyperbolic metric is used to define a K\"ahler metric on $\modm_{g,n}$ with K\"ahler form $\omega^{\text WP}_{\mathbf{L}}$.  The relationship of $\omega^{\text WP}_{\mathbf{L}}$ with the curvature of a Hermitian metric on the relative log-canonical line bundle over the universal curve generalises Wolpert's proof as follows.  Given any family of curves $X\to S$ with $n$ sections having image divisor $D=\displaystyle\sqcup_{j=1}^n D_j\subset X$, fix a choice of weights $\mathbf{a}=(a_1,...,a_n)\in(0,1]^n$ and define $\mathbf{a}\cdot D=\sum a_jD_j$.  The relative log-canonical line bundle $K_{X/S}(D)$ is equipped with a Hermitian metric, using the conical hyperbolic metric on fibres determined by $\mathbf{a}\cdot D$ for $\mathbf{a}=\mathbf{a}(\boldsymbol{\theta})$.  It has curvature, which defines it first Chern form adjusted by the weights $\mathbf{a}$, given by the real $(1,1)$ form defined over $X\setminus D$ by
\[\Omega_{X/S}(\mathbf{a}):=\tfrac{i}{2}\partial\bar{\partial}\log(g_{\mathbf{a}})\]
where $g_{\mathbf{a}}$ defines the conical hyperbolic metric $g_{\mathbf{a}}|dz|^2$ on the fibre, with cone angles $\boldsymbol{\theta}$, and the operators $\partial=\frac{\partial}{\partial z}$ and $\bar{\partial}=\frac{\partial}{\partial\bar{z}}$ use the same local coordinate $z$ on the fibre.  
\begin{thrm}[\cite{STrWei}]  \label{STrThrm} 
\begin{equation}  \label{intfibre}
\frac12\int_{X/S}\Omega_{X/S}(\mathbf{a})^2=\omega^{\text WP}_{\mathbf{L}},\quad \mathbf{a}=\mathbf{a}(-i\boldsymbol{L}).
\end{equation}
\end{thrm}
The following theorem shows that the Weil-Petersson form $\omega^{\text WP}_{\mathbf{L}}$ extends to the compactification, and hence it defines a cohomology class on Hassett's compactifications.
\begin{thrm}[\cite{ANoVol}]\label{comsumpi}
Given $\mathbf{L}=(L_1,...,L_n)\in\{i\hspace{.5mm}[0,2\pi)\}^n$ with $\modm_{g,n}^{\text{hyp}}(\mathbf{L})$ non-empty, then the Weil-Petersson form $\omega^{\text{WP}}_{\mathbf{L}}$ extends to the natural compactification $\overline{\modm}_{g,n}^{\text{hyp}}(\mathbf{L})$ by nodal hyperbolic surfaces which fits into the following commutative diagram:
\begin{equation}  \label{homcomwt}
\begin{array}{ccc}\modm_{g,n}&\stackrel{\cong}{\longrightarrow}&\modm_{g,n}^{\text{hyp}}(\mathbf{L})\\
\downarrow&&\downarrow\\
\overline{\modm}_{g,\mathbf{a}}&\stackrel{\cong}{\longrightarrow}&\overline{\modm}_{g,n}^{\text{hyp}}(\mathbf{L})
\end{array}
\end{equation}
where $\overline{\modm}_{g,\mathbf{a}}$ is the Hassett compactification with weights defined by $a_j=a(\theta_j)=a(-iL_j)$.
\end{thrm}

%%%%%%%%%%%%%%%%%%%%%%%%%%%%%%%%%%%%%%%
\subsection{The class of the Weil-Petersson form}
On %the extension of the relative log-canonical line bundle $K_{X/S}(D)$ to 
a compactification $\overline{X}/\overline{S}$ by a family of stable curves of $\pi:X\to S$, the $(1,1)$ form $\Omega_{X/S}(\mathbf{a})$ represents the following cohomology class:
\[ 2\pi\omega_\pi(\mathbf{a}\cdot D) :=2\pi(\omega_\pi+\mathbf{a}\cdot D)\in H^2(\overline{\modm}_{g,\cc},\br)
\]
where $\omega_\pi$ is the first Chern class of the relative canonical line bundle over $\overline{X}$.

Then
\begin{eqnarray*}
(2\pi\omega_\pi(\mathbf{a}\cdot D))^2&=&4\pi^2(\omega_\pi(D)+(\bold{a}-1)\cdot D)^2\\
&=&4\pi^2\left[\omega_\pi(D)^2+2(\bold{a}-1)\cdot D\cdot\omega_\pi(D)+((\bold{a}-1)\cdot D)^2\right],  
\end{eqnarray*}
and hence
\begin{eqnarray*}
\frac{1}{2}\pi_*((2\pi\omega_\pi(\mathbf{a}\cdot D))^2)&=&2\pi^2\big[\frac{1}{2}\pi_*(\omega_\pi(D)^2)+2\sum_{\cc(i,j)=0} (a_i+a_j-2)D_{i,j}\\
&&\quad-\sum(1-a_i)^2\psi_i+\sum_{\cc(i,j)=0}(1-a_i)(1-a_j)D_{i,j}\big]\\
&=&2\pi^2\big[\kappa_1-\sum (1-a_j)^2\psi_j+\sum_{\cc(i,j)=0}(2a_ia_j-2)D_{i,j}\big]
\end{eqnarray*}
which uses $\kappa_1:=\pi_*([\omega_\pi(D)]^2)$, $\pi_*(D_j\cdot\omega_\pi)=\psi_j$ and
\begin{eqnarray*}
\pi_*(D_i\cdot D_j)&=&\begin{cases}
-\psi_j&\text{ for $i=j$}\\
0&\text{ for $i\ne j$ and $a_i+a_j>1$}\\
D_{i,j}&\text{ for $i\ne j$ and $a_i+a_j\leq1$}. 
\end{cases}
\end{eqnarray*}

We use the push-forward to define a class $\gamma(\bold{a})$:
$$\frac{1}{2}\pi_*(2\pi\omega_\pi(\mathbf{a}\cdot D))^2)=2\pi^2\gamma(\bold{a}).$$
\begin{definition}  \label{def:gamma}
Define $\gamma(\bold{a})\in H^2(\Mbar{g}{\cc},\bq)[\bold{a}]$ by
\begin{equation}  \label{gama}
\gamma(\bold{a}):=\kappa_1-\sum (1-a_j)^2\psi_j+\sum_{\cc(i,j)=0}(2a_ia_j-2)D_{i,j}.
\end{equation}
When is it unclear what chamber we are working in, we write $\gamma_{\cc}(\b{a})$. 
\end{definition}
In \cite[Equation (1.2)]{ETu2Dd} a version of this formula over the Deligne-Mumford compactification is derived via a naturality and uniqueness argument, producing an element of $H^2(\overline{\modm}_{g,n},\bq)[\bold{a}]$.  Their formula involves a sum over all $S$ such that $\cc(S)=0$, whereas \eqref{gama} involves only those terms with $|S|=2$.    The natural map $\overline{\modm}_{g,n}\to\Mbar{g}{\cc}$ allows one to work over the Deligne-Mumford compactification, and inside the subring $H^*(\Mbar{g}{\cc})\to H^*(\overline{\modm}_{g,n})$.  The image of the formula \eqref{gama} under this map into $H^*(\overline{\modm}_{g,n})$ produces the formula \cite[Equation (1.2)]{ETu2Dd} (up to extra $\psi$  classes representing geodesic boundary components that can also be included here).  This is a consequence of Lemma~\ref{lem:pullbackgamma} below.
\begin{definition}  \label{def:vol}
Given a chamber $\cc\subset\cD_{g,n}$, define the volume polynomial $V_{g,\cc}$ by 
\begin{equation}  \label{volalg}
V_{g,\cc}(i\b{\theta}):=\int_{\Mbar{g}{\cc}}\exp(2\pi^2\gamma(\bold{a}))\in\br[\theta_1,...,\theta_n]
\end{equation}
where $\b{a}_j=a(\theta_j)$. 
\end{definition}
From \eqref{intfibre}, we have 
\[V_{g,\cc}(i\b{\theta})=\mathrm{Vol}\big(\cM^{\text{hyp}}_{g,n}(i\boldsymbol{\theta})\big),\quad \text{when }\b{a}(\b{\theta})\in\cc.\] 
The formula \eqref{volalg} agrees with Mirzakhani's formula \cite{MirWei} when $\cc=\cc^M$ and generalises it to all chambers.  Although the polynomial is defined for $\b{a}(\b{\theta})\not\in\cc$, evaluation outside of the chamber will not, in general, return the volume of a moduli space.

The top degree terms of $V_{g,\cc}(i\b{\theta})$ come from the top degree terms:
\[2\pi^2\gamma(\bold{a}(\bold{\theta}))=-\tfrac12\sum \theta_j^2\psi_j+\sum_{\cc(i,j)=0}\theta_i\theta_jD_{i,j}+\text{lower order terms}\]
hence involve only $\psi$ classes and the divisors $D_{i,j}$.  We expect that the top degree part of the wall-crossing should be related to the wall-crossing formulae for descendant potentials found in \cite{BCaWal}.

We conclude this section by recording how $\gamma(\bold{a})$ transforms under pullbacks of reduction and forgetful morphisms.

\begin{lemma}\label{lem:pullbackgamma}
If $\pi:\Mbar{g}{\cc}\longrightarrow \Mbar{g}{\cc'}$ is a reduction morphism with reduction sets $\mathcal{S}$ then
$$\pi^*\gamma(\bold{a})=\gamma(\bold{a})+\sum_{J\in\mathcal{S}}(1-\sum_{j\in J}a_j)^2\delta_{0:J}.$$
If $\varphi:\Mbar{g}{\cc}\longrightarrow \Mbar{g}{\cc'}$ is the forgetful morphism that forgets the last point $p_n$ with contraction sets $\mathcal{T}$, then
$$\varphi^*\gamma(\bold{a})= \gamma(\bold{a})-\sum_{\cc(i,n)=0}2a_ia_nD_{i,n}+a_n(a_n-2)\psi_n+\sum_{S\in\mathcal{T}}(1-\hspace{-3mm}\sum_{j\in S-\{n\}}\hspace{-2mm}a_j)^2 \delta_{0:S}.       $$
\end{lemma}

\begin{proof}
This follows from the definition of $\gamma(\bold{a})$ and Lemmas~\ref{lem:kappapullback1} and~\ref{lem:kappapullback2}. 
\end{proof}

%%%%%%%%%%%%%%%%%%%%%%%%%%%%%%%%%%%%%
%%%%%%%%%%%%%%%%%%%%%%%%%%%%%%%%%%%%%
\section{Proofs of main results}   \label{main}
In this section, we prove the three main theorems of the paper and the immediate geometric consequences. Theorem~\ref{g,2 volume} and its generalisation Theorem~\ref{wall-crossing formula} provide wall-crossing formulas which, in particular, recover the volume polynomial in any chamber from the volume in the main chamber given by Mirzakhani's volume recursion formula \cite{MirSim}.
% is enough to compute volumes for $\cM^{\text{hyp}}_{g,n}(i\b{\theta})$ for any set of admissible cone points $\b{\theta}\in [0,2\pi)^n$ such that $2g-2+n-\sum_{j\in\b{n}}\frac{\theta_j}{2\pi}>0$. \\ \indent 
Theorem~\ref{Thm:Dilaton} concerns the $2\pi$ limit of a cone angle. %This is done in the case where certain cone points are light, so such an evaluation corresponds geometrically to a limit as a cone point angle approaches $2\pi$. Together 

%All three of Theorems~\ref{g,2 volume},~\ref{wall-crossing formula} and~\ref{Thm:Dilaton} allow only for moduli spaces with purely cone point and cuspidal boundary components. This is done to maximise readability of formulas by sticking to the main ideas of this paper. In fact, the main theorems mostly still hold when any number of boundary geodesics, taking on $0$-weight, are added into the formulas. The only condition needed is that the lengths of each boundary component be sufficiently small, in line with \cite[Lemma 2.3]{ANoVol}.

\subsection{Wall-crossing}
Any chamber $\cc\subset\cd_{g,n}$ incident to and above a wall $W_S$ uniquely determines a chamber $\cc'$ incident to and below $W_S$. We define the wall-crossing polynomial $\wc_{\cc,S}\in\br[\theta_1,...,\theta_n]$ by
\[ 
V_{g,\cc'}(i\boldsymbol{\theta}) = V_{g,\cc}(i\boldsymbol{\theta})+\wc_{\cc,S}(\boldsymbol{\theta}).
\]
Hence a wall-crossing polynomial is the difference between the volume polynomials of the two incident chambers. 

%(If $\cc_+$ lies above the wall $W_J$, i.e. points of $\cc_+$ satisfy $\sum_{j\in J}a_j>1$, so points of $\cc_-$ satisfy $\sum_{j\in J}a_j<1$ then $WC_{\cc_+,\cc_-}(i\boldsymbol{\theta})$ is positive on the chamber $\cc_+$.
%For all $\theta$ in a fixed chamber, the moduli spaces $\Mbar{g}{\bold{a(\theta)}}$ are isomorphic. 
Denote by $\pi$ the reduction morphism between the moduli spaces specified by $\cc$ and $\cc'$. Evaluating the required top intersection numbers in the moduli space specified by the upper chamber, we obtain via Lemma~\ref{lem:pullbackgamma} the following expression for the wall-crossing:
\begin{eqnarray*}
 \wc_{\cc,S}(\boldsymbol{\theta})%&=& \int_{\Mbar{g}{\cc}}\exp(2\pi^2\pi^*\gamma(\bold{a})) - \int_{\Mbar{g}{\cc}}\exp(2\pi^2\gamma(\bold{a})) \\
 &=&\frac{(2\pi^2)^d}{d!}\int_{\Mbar{g}{\cc}}\left((\pi^*\gamma(\b{a}))^d-\gamma(\b{a})^d\right)\\
 &=& \frac{(2\pi^2)^d}{d!}\int_{\Mbar{g}{\cc}}\left((\gamma(\b{a})+(1-\sum_{j\in S}a_j)\delta_{0:S})^d-\gamma(\b{a})^d\right)
\end{eqnarray*}
where $d=3g-3+n$. To ease the notation we will sometimes omit the integral in denoting the top intersection number of a divisor. 

A wall $W_S$ may be incident to many chambers and the wall-crossing is in fact dependent on the incident chambers and not just the wall $W_S$. In \S\ref{ex:dependentwc} we provide an example where the wall-crossing polynomial is dependent on the incident chamber.
%, that is, the wall-crossing polynomial is dependent on where the wall is crossed. 
While in Corollary~\ref{cor:quotchambercrossing} and Proposition~\ref{prop:chambquotequiv} we give conditions under which the wall-crossing polynomials from different chambers incident to the same wall are equal hence depend only on the wall.

%%%%%%%%%%%%%%%%%%%%%%%%%%%%%%%
%%%%%%%%%%%%%%%%%%%%%%%%%%%%%%%%%%
\subsection{Wall-crossing formula}  
In this section we prove Theorem~\ref{g,2 volume} and Theorem~\ref{wall-crossing formula} and discuss the consequences. Though Theorem~\ref{g,2 volume} is a special case of Theorem~\ref{wall-crossing formula}, the proof of the Theorem~\ref{g,2 volume} is less technical and hence perhaps helpful to the reader.
%Interested readers may wish skim the proof of the simple case first.

%\ThmOne*
\begin{customthm}{1}
 Let $g\in\bbZ_{>0}$ and $\theta_1,\theta_2\in(0,2\pi)$ be such that $\theta_1+\theta_2> 2\pi$. Then the volume of $\cM^{\text{hyp}}_{g,2}(i\theta_1,i\theta_2)$ is a polynomial in $\theta_1,\theta_2$ satisfying
    \begin{equation*}
        \mathrm{Vol}\big(\cM^{\text{hyp}}_{g,2}(i\theta_1,i\theta_2)\big) = V^{\mathrm{Mirz}}_{g,2}(i\theta_1,i\theta_2)+\int_0^\phi V^{\mathrm{Mirz}}_{g,1}(i\theta)\cdot\theta d\theta,
    \end{equation*}
    where $\phi=\theta_1+\theta_2-2\pi$. 
\end{customthm}

\begin{proof}
  The moduli space $\cM^{\text{hyp}}_{g,2}(i\theta_1,i\theta_2)$ has a single wall $W_S$ corresponding to $S=\{1,2\}$, where the weights satisfy $a_1+a_2=1$ for $a_j=a(\theta_j)$. The volume polynomials are defined in \eqref{volalg} as top intersections in $\Mbar{g}{\cc}$, for some chamber $\cc\subset\cD_{g,n}$. Hence the wall-crossing polynomial $ \wc_{\cc^M,S}=\wc_{\cc^M,\{1,2\}}(\boldsymbol{\theta})$ is obtained from the difference of these top intersections in $\Mbar{g}{\cc^M}\cong\Mbar{g}{2}$. Let $d=3g-3+n$, 
 \begin{align} \begin{split}
   &\wc_{\cc^M,\{1,2\}}(\theta_1,\theta_2)  \\
   &=\frac{(2\pi^2)^{d}}{d!}\int_{\overline{\cM}_{g,2}}\left( \Big(\kappa_1-\sum_{j=1,2}\left(1-a_j\right)^2 \psi_j+(1-a_1-a_2)^2\delta_{0:\{1,2\}} \Big)^d \right. \\
   &\left. \hspace{5.75cm}-\Big(\kappa_1-\sum_{j=1,2}\left(1-a_j\right)^2 \psi_j\Big)^{d} \right). \label{step 1} \end{split}
\end{align}
Expand the first term of the integrand and cancel with the second term to get, \begin{equation*}
    \sum_{k=1}^{d}{d \choose k}\Big(\kappa_1-\sum_{j=1,2}\left(a_j-1\right)^2 \psi_j\Big)^{d-k}\big((1-a_1-a_2)^2\delta_{0:\{1,2\}}\big)^k.
\end{equation*}
Using $\psi_j\cdot\delta_{0:\{1,2\}}=0$ for $j=1,2$, this reduces to, \begin{equation} \label{step 2}
    \sum_{k=1}^{d}{d \choose k}\kappa_1^{d-k}\delta_{0:\{1,2\}}^k\,(1-a_1-a_2)^{2k}.
\end{equation}
In particular, we see that the wall-crossing polynomial depends only on $1-a_1-a_2$. \\ \indent 
Put \eqref{step 1} and \eqref{step 2} together to get, \begin{align}
     \wc_{\cc^M,\{1,2\}}(\b{\theta})&=\frac{(2\pi^2)^{d}}{d!}\sum_{k=1}^d{d \choose k}\int_{\overline{\cM}_{g,2}}\kappa_1^{d-k}\delta_{0:\{1,2\}}^k\,(1-a_1-a_2)^{2k} \nonumber \\
    &=\frac{(2\pi^2)^{d}}{d!}\sum_{k=1}^d{d \choose k}\int_{\overline{\cM}_{g,1}}(-1)^{k-1}\kappa_1^{d-k}\psi_1^{k-1}\,(1-a_1-a_2)^{2k}  \label{line 2} \\ 
    &=\frac{(2\pi^2)^{d}}{d!}\sum_{k=0}^{d-1}{d \choose k+1}\int_{\overline{\cM}_{g,1}}(-1)^{k}\kappa_1^{d-1-k}\psi_1^{k}\,(1-a_1-a_2)^{2(k+1)} \nonumber \\
    &=\sum_{k=0}^{d-1}K'_k(1-a_1-a_2)^{2(k+1)}, \nonumber
\end{align}
where the constants $K'_k$ are given by $K'_k=(-1)^k\frac{(2\pi^2)^{d}}{d!}{d \choose k}\kappa_1^{d-1-k}\cdot\psi_1^{k}$. Here, in \eqref{line 2} we have used $\overline{\cM}_{g,1}\cong\delta_{0:\{1,2\}}$, and so $\kappa_1^{d-k}\cdot\delta_{0:\{1,2\}}^k$ is equal to the top intersection in $\overline{\cM}_{g,1}$ of the divisor $\kappa_1^{d-k}\cdot(-\psi_1)^{k-1}$. Explicitly in terms of angles $\theta_1$ and $\theta_2$,  the wall-crossing polynomial is \begin{equation*}
     \wc_{\cc^M,\{1,2\}}(\theta_1,\theta_2)=\sum_{k=0}^{d-1}K_k(\theta_1+\theta_2-2\pi)^{2(k+1)},
\end{equation*}
where $K_k= \frac{K'_k}{ (2\pi)^{2(k+1)} }$. \\ \indent 
In what follows we find an expression for the intersection numbers $\kappa_1^{d-k}\cdot\delta_{0:\{1,2\}}^k$, in terms of $V^{\text{Mirz}}_{g,1}(i\theta)$ for $\theta\in(0,2\pi)$. By definition, \begin{align*}
    V^{\text{Mirz}}_{g,1}(i\theta):=&\frac{1}{(d-1)!}\int_{\overline{\cM}_{g,1}}\left(2\pi^2\kappa_1-\frac{1}{2}\theta^2\psi_1\right)^{d-1} \\ 
    =&\frac{1}{(d-1)!}\sum_{l=0}^{d-1}{d-1 \choose l}\left(2\pi^2\kappa_1\right)^{d-1-l}\cdot\left(-\frac{1}{2}\theta^2\psi_1\right)^l, \\
    =&\frac{1}{(d-1)!}\sum_{l=0}^{d-1}(-1)^l\frac{(2\pi^2)^{d-1-l}}{2^l}{d-1 \choose l}\kappa_1^{d-1-l}\cdot\psi_1^{l}\,\theta^{2l}.
\end{align*}
Since $$\kappa_1^{d-1-l}\cdot\psi_1^{l}=(-1)^l\frac{d!}{(2\pi^2)^{d}}\frac{1}{{d \choose l}}K_l,$$ it follows that 
\begin{align*}
    V^{\text{Mirz}}_{g,1}(i\theta)&=\sum_{l=0}^{d-1}\frac{(l+1)(2\pi^2)^{l+1}}{2^l}K'_l\,\theta^{2l}, \\
    &=\sum_{l=0}^{d-1}2(l+1)K_l\,\theta^{2l}.
\end{align*}
So $\wc_{\cc^M,\{1,2\}}(\theta_1,\theta_2)=\cL(V^{\text{Mirz}}_{g,1}(i\theta))$, where $\cL$ is a linear transformation that acts on polynomials by,\begin{equation*}
    \cL(\theta^{2l})=\begin{cases} 
      \frac{(\theta_1+\theta_2-2\pi)^{2(l+1)}}{2(l+1)} & l> 0 \\
      0 & l=0. 
   \end{cases}
\end{equation*}
This linear transformation is equivalent to integration, $\cL(\cdot)=\int_0^{\theta_1+\theta_2-2\pi}\theta d\theta(\cdot)$. So \begin{equation}
   \wc_{\cc^M,\{1,2\}}(\theta_1,\theta_2)=\int_0^{\theta_1+\theta_2-2\pi}V^{\text{Mirz}}_{g,1}(i\theta)\cdot\theta d\theta.
    \end{equation}
This gives the desired result. 
\end{proof}
In the limit as a cone angle approaches $2\pi$ we observe a particular case 
of Theorem~\ref{Thm:Dilaton}.
\begin{corollary} \label{g,2 volume corollary}
Let $g\in\bbZ_{>0}$. In the $2\pi$ cone angle limit, the $n=2$ volumes satisfy:
    \begin{align*} 
          &\lim_{\theta_2\to 2\pi} \mathrm{Vol}\big(\cM^{\text{hyp}}_{g,2}(i\theta_1,i\theta_2)\big)=0,\hspace{0.2cm} \text{and} \\
          &\lim_{\theta_2\to 2\pi}\frac{\partial}{\partial \theta_2}   \mathrm{Vol}\big(\cM^{\text{hyp}}_{g,2}(i\theta_1,i\theta_2)\big) =2\pi \left(1-2g+\tfrac{\theta_1}{2\pi}\right)V^{\text{Mirz}}_{g,1}(i\theta_1).
    \end{align*}
    %where $\chi_{g,n}(\theta_1,\ldots,\theta_n)= 2-2g-n+\sum_{j=1}^n\frac{\theta_j}{2\pi}$ is the general form for the Euler characteristic of a surface with $n$ cone points. NOTE: the Gauss-Bonnet formula separates angles and Euler characteristic
\end{corollary}
\begin{proof}
From \eqref{limitdil1} we have 
\[V^{\text{Mirz}}_{g,2}(L_1,2\pi i)=\int_0^{L_1}LV^{\text{Mirz}}_{g,n}(L)dL\Rightarrow V^{\text{Mirz}}_{g,2}(i\theta_1,2\pi i)=-\int_0^{\theta_1}\theta V^{\text{Mirz}}_{g,n}(i\theta)d\theta.
\]
Hence by Theorem~\ref{g,2 volume}
\[\mathrm{Vol}\big(\cM^{\text{hyp}}_{g,2}(i\theta_1,2\pi i)\big) = -\int_0^{\theta_1}\theta V^{\text{Mirz}}_{g,n}(i\theta)d\theta+\int_0^\phi V^{\mathrm{Mirz}}_{g,1}(i\theta)\cdot\theta d\theta=0
\]
since $\phi=\theta_1+\theta_2-2\pi\stackrel{\theta_2=2\pi}{\longmapsto}\theta_1$.
Similarly, from \eqref{limitdil2} we have 
\[\frac{\partial V^{\text{Mirz}}_{g,2}}{\partial L_{2}}(L_1,2\pi i)=2\pi (2g-1)iV^{\text{Mirz}}_{g,n}(L_1)\]
hence
\[\frac{\partial V^{\text{Mirz}}_{g,2}}{\partial \theta_{2}}(i\theta_1,2\pi i)=2\pi(1-2g)V^{\text{Mirz}}_{g,n}(i\theta_1).
\]
Using 
\[\frac{\partial}{\partial \theta_{2}}\wc_{\cc^M,\{1,2\}}(\theta_1,\theta_2)=\frac{\partial}{\partial \theta_{2}}\int_0^\phi V^{\text{Mirz}}_{g,1}(i\theta)\cdot\theta d\theta=V^{\text{Mirz}}_{g,1}(i\phi)\cdot\phi
\]
with Theorem~\ref{g,2 volume}, we have
\[\left.\frac{\partial}{\partial \theta_{2}}\right|_{\theta_2=2\pi}\hspace{-4mm}\mathrm{Vol}\big(\cM^{\text{hyp}}_{g,2}(i\theta_1,i\theta_2)\big) = 2\pi\left(1-2g+\tfrac{\theta_1}{2\pi}\right) V^{\text{Mirz}}_{g,n}(i\theta_1)
\]
as required.
\end{proof}
\begin{remark}
The proof of Corollary~\ref{g,2 volume corollary} generalises to any $n$ and can be used in reverse to prove \eqref{limitdil1} from the vanishing property \eqref{eqn:incidentzero}.
\end{remark}

%This corollary will be generalised in Theorem~\ref{Thm:Dilaton}. 

%A direct consequence of Theorem~\ref{g,2 volume} is an equation for wall-crossing polynomials $\wc_{\cc,S}$ corresponding the main chamber $\cc=\cc^M$ and to walls $W_S$ with $|S|=2$. 
The proof of the more general wall-crossing formula in Theorem~\ref{wall-crossing formula} follows similar ideas to that of Theorem~\ref{g,2 volume} albeit  more technical.

\thmtwo*

\begin{proof}
The volume polynomials are expressed in \eqref{volalg} as top intersections in $\Mbar{g}{\cc'}$, for some chamber $\cc' \subset\cD_{g,n}$. Hence the wall-crossing polynomial $\wc_{\cc,S}=\wc_{\cc,S} (\boldsymbol{\theta})$ is obtained from the difference of these top intersections in $\Mbar{g}{\cc}$. Throughout this proof we omit integral signs $\int_{\Mbar{g}{\cc}}$ and write products of cohomology classes to mean top intersection numbers of divisors.  Let $d=3g-3+n$,  
\begin{align*}
    \wc_{\cc,S}%(i\boldsymbol{\theta}(\b{a}))
    = \frac{(2\pi^2)^{d}}{d!}&\left(\Big(\kappa_1-\sum_{j=1}^n(1-a_j)^2\psi_j +(1-\sum_{j\in S}a_j)^2\delta_{0:S}\Big)^{d}\right. \\
    &\hspace{5.2cm}\left.-\Big(\kappa_1-\sum_{j=1}^n(1-a_j)^2\psi_j\Big)^{d} \right)  \end{align*}
 where $a_j=a(\theta_j)$. 
 %gives the $\boldsymbol{\theta}$ dependence implicitly.
Hence
\begin{align*}
 \wc_{\cc,S}&= \frac{(2\pi^2)^{d}}{d!} \sum_{k=1}^d\binom{d}{k} \Big(\kappa_1-\sum_{j=1}^n(1-a_j)^2\psi_j\Big)^{d-k}\left(\Big(1-\sum_{j\in S}a_j\Big)^2\delta_{0:S}\right)^k\\
 &= \frac{(2\pi^2)^{d}}{d!} \Big(1-\sum_{j\in S}a_j\Big)^2\delta_{0:S}\sum_{k=1}^d\binom{d}{k}\Big(\kappa_1-\sum_{j=1}^n(1-a_j)^2\psi_j\Big)^{d-k} \\
 &\hspace{7.2cm} \cdot\left(\Big(1-\sum_{j\in S}a_j\Big)^2\delta_{0:S}\right)^{k-1}.
\end{align*}
Now by the identification $\delta_{0:S}\cong \overline{\cM}_{0,\cc/S}\times\overline{\cM}_{0,\cc_1}$, and denoting the projections to the factors as $\pi_1$ and $\pi_2$ respectively, we denote
$$\kappa_1':=\pi_1^*\kappa_1,\hspace{1cm} \kappa_1'':=\pi_2^*\kappa_1, $$
$$\psi_j=\begin{cases}\pi_1^*\psi_j&\text{ for $j\in S^c$}\\
\pi_2^*\psi_{j)}&\text{ for $j\in S$}\end{cases},$$
$$ \psi_\alpha:=\pi_1^*\psi_{n-|S|+1},\hspace{1cm}\psi_\beta:= \pi_2^*\psi_1.$$ 
%where $\alpha$ and $\beta$ are cuspidal points which arise in the geometric picture of the projection onto products. 
Further, we observe that the normal bundle of $\delta_{0:S}$ is the class $(-1)(\psi_\alpha+\psi_\beta)$ and in an attempt to control notation we denote
$$\kappa_\alpha:=\kappa_1'-\sum_{j\in S^c\cup \{\alpha\}}(1-a_j)^2\psi_j\text{ and } \kappa_\beta:=\kappa_1''-\sum_{j\in S\cup\{\beta\}}^{n}(1-a_j)^2\psi_j.$$
Our final observation is that 
\begin{equation}\label{eqn:int}
\prod_{i=1}^{d-k}\pi_1^*D_i\cdot \prod_{i=1}^{k-1}\pi_2^*E_i=0
\end{equation}
for $k\ne |S|-1$, given any classes $D_i\in H^2(\overline{\cM}_{0,\cc/S})$ and $E_i\in H^2(\overline{\cM}_{0,\cc_1})$. 

With all this in hand, set $|a_S|=\sum_{j\in S}a_j$ and by restriction of all classes to $\delta_{0:S}$,
we obtain that $\frac{d!}{(2\pi^2)^{d}} \wc_{\cc,S}$ is equal to
\begin{align*}
&(1-|a_S|)^2\sum_{k=1}^d(-1)^k\binom{d}{k}(\kappa_\alpha+\kappa_\beta)^{d-k}\Big((1-|a_S|)^2(-1)(\psi_\alpha+\psi_\beta)\Big)^{k-1}   \\
 =&\,(1-|a_S|)^2\sum_{k=1}^d\binom{d}{k}\sum_{t=0}^{d-k}\sum_{q=0}^{k-1}\binom{d-k}{t}\binom{k-1}{q} \\
 &\hspace{5.5cm}\cdot(-1)^{k-1}\kappa_\alpha^t\kappa_\beta^{d-k-t}\psi_\alpha^q\psi_\beta^{k-q-1}(1-|a_S|)^{2(k-1)}\\
 =&\,\sum_{k=1}^d\binom{d}{k}\sum_{j=1}^{\min\{k,|S|-1\}}(-1)^{k-1}(1-|a_S|)^{2k} \\
 &\hspace{3cm}\cdot\binom{d-k}{d-s-k+j+1}\binom{k-1}{j-1}\kappa_\alpha^{d-s-k+j+1}\psi_\alpha^{k-j} \kappa_\beta^{|S|-j-1}\psi_\beta^{j-1}.
 \end{align*}
Here the last line follows from applying \eqref{eqn:int}, which implies that only possibly monomials with $t+q=d-s+1$ are non-vanishing. Rearranging we obtain 
\begin{align*}
&\sum_{k=0}^{d-1}\binom{d}{k+1}\sum_{j=1}^{\min\{k+1,|S|-1\}}(-1)^{k}(1-|a_S|)^{2(k+1)}\\
&\hspace{3.4cm}\cdot\binom{d-k-1}{d-|S|-k+j}\binom{k}{j-1}\kappa_\alpha^{d-|S|-k+j}\psi_\alpha^{k-j+1} \kappa_\beta^{|S|-j-1}\psi_\beta^{j-1}.\\
\end{align*}
Hence,
\begin{eqnarray*}
\wc_{\cc,S}&=&\sum_{k=0}^{d-1}K_k(2\pi-\sum_{j\in S}(2\pi-\theta_j))^{2(k+1)} \\
&=&\sum_{k=0}^{d-1}K_k\phi_S^{2(k+1)},
\end{eqnarray*}
where 
\begin{align*} K_k=\frac{(-1)^{k}(2\pi^2)^d}{(2\pi)^{2(k+1)}d!}\binom{d}{k+1}\sum_{j=1}^{\min\{k+1,|S|-1\}}&\binom{d-k-1}{d-|S|-k+j}\binom{k}{j-1}\\ 
&\hspace{1.1cm}\cdot\kappa_\alpha^{d-|S|-k+j}\psi_\alpha^{k-j+1} \kappa_\beta^{|S|-j-1}\psi_\beta^{j-1}.\end{align*}
We now observe that setting $\theta=2\pi(1-a)$ we have
\begin{align*}
 V_{g,\cc/S}(i\b{\theta}_{S^c},i\theta)&=\frac{(2\pi^2)^{d-|S|+1}}{(d-|S|+1)!}\left(\kappa_{\alpha}-(1-a)^2\psi_\alpha  \right)^{d-|S|+1}\\
&=\frac{(2\pi^2)^{d-|S|+1}}{(d-|S|+1)!}\sum_{k=0}^{d-|S|+1}(-1)^k(1-a)^{2k}\binom{d-|S|+1}{k}\kappa_{\alpha}^{d-k-|S|+1}\psi_\alpha^k \\
&=\sum_{k=0}^{d-|S|+1}M_k\theta^{2k},
\end{align*}
for 
$$M_k= (-1)^k\frac{(2\pi^2)^{d-|S|+1}}{(2\pi)^{2k}(d-|S|+1)!}\binom{d-|S|+1}{k}\kappa_{\alpha}^{d-k-|S|+1}\psi_\alpha^k $$
and
\begin{align*}
V_{0,\cc_1}(i\theta,i\b{\theta}_S) &=\frac{(2\pi^2)^{|S|-2}}{(|S|-2)!}\left(\kappa_{\beta}-(1-a)^2\psi_\beta  \right)^{|S|-2}\\
&=\frac{(2\pi^2)^{|S|-2}}{(|S|-2)!}\sum_{k=0}^{|S|-2}(-1)^k(1-a)^{2k}\binom{|S|-2}{k}\kappa_{\beta}^{|S|-k-2}\psi_\beta^k \\
&=\sum_{k=0}^{|S|-2}N_k\theta^{2k},
\end{align*}
for 
$$N_k= (-1)^k\frac{(2\pi^2)^{|S|-2}}{(2\pi)^{2k}(|S|-2)!}\binom{|S|-2}{k}\kappa_{\beta}^{|S|-k-2}\psi_\beta^k $$

If we define $K'_k$ by the equation
$$\sum_{k=0}^dK'_k\theta^{2(k+1)}=\int_0^\theta \left(\sum_{k=0}^{d-|S|+1}M_k t^{2k}\right)\cdot\left(\sum_{k=0}^{|S|-2}N_k t^{2k}\right)\cdot t\cdot dt$$
then it remains to show $K'_k=K_k$. We have
\begin{align*}
&\hspace{-0.05cm}K'_k %\\ &\hspace{-0.05cm}
=\frac{1}{2(k+1)}\hspace{-0.4cm}\sum_{j=0}^{\min\{k,|S|-2\}}N_j\cdot M_{k-j}\\
&\hspace{-0.05cm}=\frac{(-1)^k(2\pi^2)^{d-1}}{2(k+1)(2\pi)^{2k}}\sum_{j=0}^{\min\{k,|S|-2\}}\hspace{-0.4cm}\frac{1}{(|S|-2)!}\binom{|S|-2}{j}\kappa_{\beta}^{|S|-j-2}\psi_\beta^j  \\
 &\hspace{5.0cm}\cdot\frac{1}{(d-|S|+1)!}\binom{d-|S|+1}{k-j}\kappa_{\alpha}^{d-k+j-|S|+1}\psi_\alpha^{k-j} \\
 &\hspace{-0.05cm}=\frac{(-1)^k(2\pi^2)^{d-1}}{2(k+1)(2\pi)^{2k}}\sum_{j=1}^{\min\{k+1,|S|-1\}}\hspace{-0.4cm}
{\frac{1}{(k-j+1)!(|S|-j-1)!(j-1)!(d-|S|-k+j)!}}\\
&\hspace{7.4cm}\cdot\kappa_{\alpha}^{d-k+j-|S|}\psi_\alpha^{k-j+1} \kappa_{\beta}^{|S|-j-1}\psi_\beta^{j-1} \\
 &\hspace{-0.05cm}=\hspace{-0.03cm}\frac{(-1)^k(2\pi^2)^{d-1}}{2(2\pi)^{2k}d!}\hspace{-0.08cm}\binom{d}{k+1}\hspace{-0.2cm}\sum_{j=1}^{\min\{k+1,|S|-1\}}\hspace{-0.8cm}
{\frac{(d-k-1)!k!}{(k-j+1)!(|S|-j-1)!(j-1)!(d-|S|-k+j)!}}\\
&\hspace{7.6cm}\cdot\kappa_{\alpha}^{d-k+j-|S|}\psi_\alpha^{k-j+1} \kappa_{\beta}^{|S|-j-1}\psi_\beta^{j-1} \\
&\hspace{-0.05cm}=\frac{(-1)^k(2\pi^2)^{d-1}}{2(2\pi)^{2k}d!}\binom{d}{k+1}\hspace{-0.1cm}\sum_{j=1}^{\min\{k+1,|S|-1\}}\hspace{-0.1cm}\binom{d-k-1}{d-|S|-k+j}\binom{k}{j-1} \\ 
&\hspace{7.5cm}\cdot\kappa_{\alpha}^{d-k+j-|S|}\psi_\alpha^{k-j+1} \kappa_{\beta}^{|S|-j-1}\psi_\beta^{j-1} \\
&=K_k
\end{align*}
\normalsize
completing the proof.
\end{proof}

\begin{remark}
Theorem~\ref{g,2 volume} deals with the $S=\{1,2\}$ case, where $|S|=2$ and
\begin{equation*}
\wc_{\cc,S}(\theta_1,\theta_2)=\sum_{k=0}^{d-1}K_k\phi_{1,2}^{2(k+1)}=\sum_{k=0}^{d-|S|+1}\frac{M_k}{2(k+1)}\phi_{1,2}^{2(k+1)}.
\end{equation*}
Indeed, here for $|S|=2$ we have that
$$V_{0,\cc_1}(i\theta,i\theta_1,i\theta_2) =1$$
and
\begin{align*}
K_k &=\frac{(-1)^{k}(2\pi^2)^d}{(2\pi)^{2(k+1)}d!}\binom{d}{k+1}\binom{d-k-1}{d-k-1}\binom{k}{0}\kappa_\alpha^{d-k-1}\psi_\alpha^{k} \kappa_\beta^{0}\psi_\beta^{0} \\
&=\frac{(-1)^{k}(2\pi^2)^d}{(2\pi)^{2(k+1)}d!}\binom{d}{k+1}\kappa_\alpha^{d-k-1}\psi_\alpha^{k}    \\
&=\frac{1}{2(k+1)}\left((-1)^k\frac{(2\pi^2)^{d-1}}{(2\pi)^{2k}(d-1)!}\binom{d-1}{k}\kappa_{\alpha}^{d-k-1}\psi_\alpha^k   \right)\\
&=\frac{M_k}{2(k+1)}.
\end{align*}
\end{remark}
%\begin{remark}An explicit formula for $V_{0,\cc_1}$ volume polynomials is computed as an example, in Section~\ref{volP^n}. Here we restate Theorem~\ref{wall-crossing formula} as Theorem~\ref{thm:wall2} which incorporates this expression into the wall-crossing formula.\end{remark}
 Mirzakhani's volume recursion formula \cite{MirSim}, allows one to compute volume polynomials corresponding to the main chamber $\cc^M$ of any space of stability conditions $\cD_{g,n}$. Along with formulas for wall-crossing given by Theorem~\ref{wall-crossing formula}, one could compute volumes for $\cM^{\text{hyp}}_{g,n}(i\b{\theta})$ for any set of admissible cone points $\b{\theta}\in [0,2\pi)^n$ such that $2g-2+n-\sum_{j\in\b{n}}\frac{\theta_j}{2\pi}>0$. Although, in practice this becomes difficult for large $n$ cases when dealing with chambers that require many wall-crossings to reach the main chamber. We will see in Section~\ref{sec:ex} how to combine wall-crossing with other techniques to obtain volume polynomials in some of these cases. 

As well as a tool for explicit computation, Theorem~\ref{wall-crossing formula} tells us some general information about volume polynomials. Notice from \eqref{WCF} that a given wall-crossing polynomial $\wc_{\cc,S}$ depends only on the quotient chamber $\cc/S$. 
\begin{corollary}\label{cor:quotchambercrossing}
     Given two chambers $\cc_1$ and $\cc_2$ which are both incident to and above a wall $W_S$, their respective wall-crossing polynomials $\wc_{\cc_1,S}$ and $\wc_{\cc_2,S}$ are equal if $\cc_1/S=\cc_2/S$.
\end{corollary}
This dependence is stronger than simply specifying a wall. Geometrically, crossing a wall $W_S$ is only necessarily independent of the chambers above $S$, if those chambers are separated by a series of walls $W_{T_i}$ such that $T_i\cap S\neq \emptyset$ for all $i$. This result can also be proven directly using intersection theory and the fact that $\delta_{0:S}\cdot\delta_{0:T_i}=0$ if $T_i\cap S\neq \emptyset$. 
\begin{proposition}\label{prop:chambquotequiv}
    Given two chambers $\cc_1$ and $\cc_2$ separated by a wall $W_T$, the condition $T\cap S\neq\emptyset$ is equivalent to $\cc_1/S=\cc_2/S$.
\end{proposition}
\begin{proof}
    Suppose that $\cc_1$ and $\cc_2$ are above and below $W_T$, respectively. The chambers are related by $$\cc_2(J)=\begin{cases}\cc_1(J)&\text{ if $J=T$}\\
0&\text{ if $J\neq T$}\end{cases}.$$ 
So 
\begin{align*}
\cc_2/S(\bar{J})&=\begin{cases}
0&\text{ if $J\subset S$} \\ 
\cc_2 (J)&\text{ if $J\cap S=\emptyset$}\\
1&\text{ otherwise}\end{cases} \\
&=\begin{cases}
0&\text{ if ($J\subset S$) or ($J\cap S=\emptyset$ and $J=T$)} \\
\cc_1(J)&\text{ if $J\cap S=\emptyset$ and $J\neq T$}\\
1&\text{ otherwise}\end{cases}.
\end{align*}
If $T\cap S\neq \emptyset$ then $J\cap S=\emptyset\Rightarrow J\neq T$ so $$\cc_2/S(\bar{J})=\left.\begin{cases}
0&\text{ if $J\subset S$ } \\
\cc_1(J)&\text{ if $J\cap S=\emptyset$ }\\
1&\text{ otherwise}\end{cases}\right\}=\cc_1/S(\bar{J}).$$
Conversely, if $\cc_1/S=\cc_2/S$ then for all $\bar{J}$ it follows that $J\cap S=\emptyset\Rightarrow J\neq T$. So $T\cap S\neq \emptyset$. 
\end{proof}

The collection of polynomials $V_{g,\cc}(i\boldsymbol{\theta})$ forms a piecewise polynomial on $\cd_{g,n}$, which determines the volume $\mathrm{Vol}\big(\cM^{\text{hyp}}_{g,n}(i\b{\theta})\big)$ as cone angles vary.  Moreover, the volumes vary continuously and differentiably across walls. 
\begin{lemma}
The volume of $\cM^{\text{hyp}}_{g,n}(i\b{\theta})$ is everywhere continuous and differentiable in angles $\theta_j$. 
\end{lemma}
\begin{proof}
    Away from walls this result is trivial. What's left to show is that wall-crossing polynomials and their derivatives vanish when evaluated at the corresponding wall. Consider a wall-crossing polynomial $\wc_{\cc,S}$. At the wall $W_S$, we have $\sum_{j\in S}a_j=1$ which is equivalent to $\phi_S=0$, hence $\wc_{\cc,S}(\b{\theta})\mid_{W_S}=0$ and the volume is continuous across walls. 
    
To understand derivatives of the volume at the wall $W_S$, we consider two cases.  If $j\in S$, then the derivative with respect to $\theta_j$ is 
    \begin{align*}
        \frac{\partial}{\partial\theta_j}\wc_{\cc,S}(\b{\theta})=V_{g,\cc/S}(i\boldsymbol{\theta}_{S^c},&i\phi_S)\cdot V_{0,\cc_1}(i\phi_S,i\boldsymbol{\theta}_{S})  \cdot \phi_S \\
        & + \int_{0}^{\phi_S} V_{g,\cc/S}(i\boldsymbol{\theta}_{S^c},i\theta)\cdot  \frac{\partial}{\partial\theta_j} \left( V_{0,\cc_1}(i\theta,i\boldsymbol{\theta}_{S}) \right) \cdot \theta \cdot d\theta.
    \end{align*}
    So at the wall $W_S$, where $\phi_S=0$, it follows that $\frac{\partial}{\partial\theta_j}\wc_{\cc,S}(\b{\theta})\mid_{W_S}=0$. If $j\notin S$, then \begin{equation*}
        \frac{\partial}{\partial\theta_j}\wc_{\cc,S}(\b{\theta})=\int_{0}^{\phi_S}\frac{\partial}{\partial\theta_j} \left( V_{g,\cc/S}(i\boldsymbol{\theta}_{S^c},i\theta)\right) \cdot  V_{0,\cc_1}(i\theta,i\boldsymbol{\theta}_{S})  \cdot \theta \cdot d\theta,
    \end{equation*}
which again vanishes when $\phi_S=0$.
\end{proof}

%%%%%%%%%%%%%%%%%
\subsection{The dilaton equation and limits of volume polynomials}
%The relations \eqref{limitdil} evoke the idea of a $2\pi $ limit of a cone angle, however until now no geometric meaning has been deduced from these relations.  This is because evaluation of the polynomials does not necessarily correspond to volumes close to the limit $\theta\to 2\pi$.    %In fact, $V_{0,4}(0,L,\theta i,2\pi i)>0$ for $L\gg 0$ is also not a volume.  
%One outcome of this paper is to describe occurrences when these relations do and do not have geometric meaning. 
%Theorem~\ref{Thm:Dilaton} looks at the $2\pi i$ evaluation of volume polynomials and their derivatives. This is done in the case where certain cone points are light, so such an evaluation corresponds geometrically to a limit as a cone point angle approaches $2\pi$. Together this serves as a generalisation of Theorem 2 from \cite{DNoWei} and Corollary 1.3 from \cite{ANoVol}.

%A key result of this paper generalises even further, to volume polynomials which do not correspond to Mirzakhani's polynomials. 

%Results in \cite{ANoVol} described occurrences when the relations \eqref{limitdil} do and do not have geometric meaning. Here, we do the same with Theorem~\ref{Thm:Dilaton}. 

In this section we prove Theorem~\ref{Thm:Dilaton}, which concerns the $2\pi i$ limit of a derivative of a volume polynomial. %This result generalises work from \cite{DNoWei} by considering volume polynomials which do not correspond to Mirzakhani's polynomials.
Recall that we defined a chamber $\cc\subset\cd_{g,n}$ as being \emph{light in the $i$th coordinate} if for every set $S\subset \{1,\dots, n\}$ (including singletons) such that $i\not\in S$ and $\sum_{j\in S} a_j\leq1$, it also holds that $a_i+\sum_{j\in S} a_j\leq1$. Or equivalently, the contraction set is empty for the forgetful morphism forgetting the $i$th point.

\begin{lemma}
 \label{Thm:2pilimits}
Let $\cc\subset\cd_{g,n+1}$ be light in the $(n+1)$th coordinate. Then
\begin{equation}
V_{g,\cc}(i\theta_1,\dots,i\theta_{n},2\pi i)=0.
\end{equation}
\end{lemma}
\begin{proof} We denote by 
$$\varphi:\Mbar{g}{\cc}\longrightarrow \Mbar{g}{\cc'}$$ 
the forgetful morphism which is flat by assumption, but further, as $\cc$ is light in the $(n+1)$th coordinate Lemma~\ref{lem:pullbackgamma} gives
$$ \gamma_\cc(\bold{a})=\varphi^*\gamma_{\cc'}(\bold{a})+a_{n+1}(2-a_{n+1})\psi_{n+1}+2\sum_{j=1}^na_ja_{n+1}D_{j,n+1},$$
where $\gamma_{\cc}$ and $\gamma_{\cc'}$ are defined in Definition~\ref{def:gamma}, and depend on their respective chambers. 
Letting $d=3g-3+(n+1)$ we obtain by the flatness of $\varphi$
$$\lim_{a_{n+1}\to 0}\gamma_{\cc}(\bold{a})^d=(\varphi^*\gamma_{\cc'}(\bold{a}))^d=0.$$
Hence \eqref{eqn:incidentzero} holds.

Alternatively, the K\"ahler metric is proven in \cite[Thm 3.5]{STrVar} to degenerate in the $2\pi$ limit of a cone angle to the pullback of the metric on the moduli space with the point forgotten.  If uniform convergence to the pullback of the metric can be proven then a consequence would be that the limit
 \begin{equation}  \label{limit}
 \lim_{\theta_{n+1}\to 2\pi i}V_{g,n}(i\theta_1,i\theta_2,...,i\theta_{n+1})=0
\end{equation}
This limit, together with the polynomial dependence of the volume, would give a geometric proof of \eqref{eqn:incidentzero}.  %We omit the details, due to the algebro-geometric proof above.
\end{proof}

Recall that we defined a chamber $\cc\subset\cd_{g,n}$ as being \emph{flat in the $i$th coordinate} if the contraction set contains only sets of order at most $2$ for the morphism forgetting the $i$th point. for the forgetful morphism forgetting the $i$th point. This is a weaker condition that the chamber being light in the $i$th coordinate.

\thmfour*
We can equivalently write $q(\cc)$ as follows:
\begin{eqnarray*}
q(\cc)=\{j\mid \cc(\{j,n+1\})=1\}.
\end{eqnarray*}
\begin{proof}
Denote by 
$$\varphi:\Mbar{g}{\cc}\longrightarrow \Mbar{g}{\cc'}$$ 
the forgetful morphism which is flat by assumption. Denote by $\gamma_{\cc}(\b{a})$ and $\gamma_{\cc'}(\b{a})$ the cohomology classes defined in Definition~\ref{def:gamma}, which depend on their respective chambers. Throughout this proof we write products of cohomology classes as shorthand notation for top intersections.  

Observe that we can obtain
$$\frac{\partial(\gamma_{\cc}(\bold{a})^d)}{\partial a_{n+1}}\bigm|_{a_{n+1}=0}$$
as the coefficient of $a_{n+1}$ in the polynomial expansion of 
$$(\varphi^*\gamma_{\cc'}(\bold{a})+a_{n+1}(2-a_{n+1})\psi_{n+1}+2\sum_{j\notin q(\cc)} a_ja_{n+1}D_{j,n+1} -\sum_{j\in q(\cc)}(1-a_j)^2\delta_{0:\{j,n+1\}})^d   $$
which is
$$d\cdot (\varphi^*\gamma_{\cc'}(\bold{a}) -\sum_{j\in q(\cc)}(1-a_j)^2\delta_{0:\{j,n+1\}})^{d-1}(2\psi_{n+1}+2\sum_{j\notin q(\cc)}a_jD_{j,n+1} )    $$
But as $\psi_{n+1}\cdot \delta_{0:\{j,n+1\}}=D_{k,n+1}\cdot \delta_{0:\{j,n+1\}}=0$ for $k\notin q(\cc)$ and $j\in  q(\cc)$ this becomes
$$d\cdot (\varphi^*\gamma_{\cc'}(\bold{a}) )^{d-1}(2\psi_{n+1}+2\sum_{j\notin q(\cc)}a_jD_{j,n+1} )    $$
and as $\varphi$ is flat we have
$$(\varphi^*\gamma_{\cc'}(\bold{a}))^{d-1}=\varphi^*(\gamma_{\cc'}(\bold{a})^{d-1})=\deg[\gamma_{\cc'}(\bold{a})^{d-1}] [B]$$
where $[B]$ is the curve with class equal to the fibre of $\varphi$ and hence by the intersection numbers
$$B\cdot\psi_{n+1}=2g-2+|q(\cc)|\text{  and  } B\cdot D_{i,n+1}=1.    $$
Now
\begin{eqnarray*}
\frac{\partial(\gamma_{\cc}(\bold{a})^d)}{\partial a_{n+1}}\bigm|_{a_{n+1}=0}&=&2d(2g-2+|q(\cc)|+\sum_{j\notin q(\cc)}a_j)\cdot\deg[\gamma_{\cc'}(\bold{a})^{d-1}]\\
&=&4\pi^2(2g-2+|q(\cc)|+\sum_{j\notin q(\cc)} a(\theta_j))V_{g,\cc'}(i\theta_1,\dots,i\theta_n).
\end{eqnarray*}
Then \eqref{eqn:derivative} follows by the chain rule as
$$\frac{da(\theta_{n+1})}{d\theta_{n+1}}=\frac{-1}{2\pi}.$$
\end{proof}
\begin{remark}
     In Corollary~\ref{cor:generalchamber} we use wall-crossing polynomials and results from Section~\ref{volP^n} to obtain a generalisation of \eqref{eqn:derivative} that holds for certain chambers that are not necessarily flat in the $(n+1)$th coordinate.
\end{remark}
Lemma~\ref{Thm:2pilimits} and Theorem~\ref{Thm:Dilaton}  serve as a generalisation of Theorem 2 from \cite{DNoWei} and Corollary 1.3 from \cite{ANoVol}. The proof of Theorem~\ref{Thm:Dilaton} also applies to cases where boundary geodesics are present provided the lengths of the added boundary component are sufficiently small (see \cite[Lemma 2.3]{ANoVol}).  Each boundary geodesic is assigned the weight of a cusp in order to determine the appropriate chamber. 

%\begin{remark}We will see in Corollary \ref{cor:wallderiv}, that in some cases the derivatives of wall-crossing %, which appear in \eqref{eqn:derivativegeneral}, are themselves wall-crossings between certain restricted chambers.\end{remark}

The two extremal cases of Theorem~\ref{Thm:Dilaton}, for $q(\cc)=\b{n}$ and $q(\cc)=\emptyset$, have special significance.   When $a_j+a_{n+1}>1$ for all $j\in \b{n}$, for example in the maximal chamber $\cc^M$, so that $q(\cc)=\b{n}$, then \eqref{eqn:derivative} becomes 
\begin{equation} \label{eqn:dilMirz}
    \frac{\partial V_{g,\cc}}{\partial\theta_{n+1}}(i\theta_1,\dots,i\theta_n,2\pi i)=-2\pi (2g-2+n)V_{g,\cc|_{\b{n}}}(i\theta_1,\dots,i\theta_n).
\end{equation} 
When applied to $\cc^M$, the volume polynomials are Mirzakhani's polynomials, and \eqref{eqn:dilMirz} is the dilaton equation proven in \cite{ANoVol}. 

When $\cc\subset\cD_{g,n+1}$ is light in the $(n+1)$th coordinate, for example when $a_{n+1}\approx 0$,  then $q(\cc)=\emptyset$, and \eqref{eqn:derivative} becomes 
\begin{equation} \label{eqn:dillight}
    \frac{\partial V_{g,\cc}}{\partial\theta_{n+1}}(i\theta_1,\dots,i\theta_n,2\pi i)=-2\pi (2g-2+\sum_{j=1}^n a_j)V_{g,\cc|_{\b{n}}}(i\theta_1,\dots,i\theta_n).
\end{equation}
The coefficient in \eqref{eqn:dillight} is given by the negative of the hyperbolic area 
\[A=2\pi (2g-2+\sum_{j=1}^n a_j).\]  
We expect that this coefficient should be a consequence of the degeneration of the K\"ahler metric in the $2\pi$ limit of a cone angle to the pullback of the metric from the moduli space with the point forgotten, proven in \cite[Thm 3.5]{STrVar}, once the rate of degeneration is proven.

In general, $2\pi i$ evaluation of a volume polynomial $V_{g,\cc}$ or its derivative does not necessarily correspond to volumes of moduli spaces $\cM^{\text{hyp}}_{g,n+1}(i\b{\theta},i\theta_{n+1})$ in the limit $\theta_{n+1}\to 2\pi$. This is because a path taken in such a limit likely crosses a number of walls within the space of stability conditions $\cD_{g,n+1}$. However, in the case where volume polynomials correspond to chambers where the $(n+1)$th coordinate is light, evaluation of a light point to $\theta_{n+1}=2\pi i$ does have geometric meaning, since a path in $\cD_{g,n+1}$ sending the light point cone angle to $2\pi$ remains in the same chamber. %This means that equation \eqref{eqn:dillight} has geometric meaning. 
\iffalse
\begin{corollary} \label{cor:geo}
   In the case where $\cc\subset\cD_{g,n+1}$ is light in the $(n+1)$th coordinate Theorem~\ref{Thm:Dilaton} has geometric meaning. That is, \eqref{eqn:dillight} holds when volume polynomials are replaced with true volumes of moduli spaces under $2\pi$ cone angle limits, \begin{align*}
        V_{g,\cc}(i\b{\theta})&=\mathrm{Vol}\big(\cM^{\text{hyp}}_{g,n+1}(i\b{\theta})\big) \\
        V_{g,\cc'}(i\b{\theta})&=\mathrm{Vol}\big(\cM^{\text{hyp}}_{g,n}(i\b{\theta})\big) \\
        V_{g,\cc}(i\b{\theta},2\pi i)&=\lim_{\theta_{n+1}\to 2\pi}\mathrm{Vol}\big(\cM^{\text{hyp}}_{g,n}(i\boldsymbol{\theta},i\theta_{n+1})\big).
    \end{align*}
\end{corollary}
\fi

The following calculations demonstrate the difference between the $2\pi$ limit of the derivative of the piecewise polynomial $\mathrm{Vol}\big(\cM^{\text{hyp}}_{g,n}(i\theta_1,i\theta_2,i\theta_3,i\theta_4)\big)$ versus the same limit of the derivative of the polynomial in a flat chamber.   Consider the chamber $\cB_2\subset\cD_{0,4}$ from Section~\ref{ex:M04}, which is flat in the $4$th coordinate. This chamber is incident to the wall $W_{\{1,4\}}$ and above the chamber $\cb_3$ which is light in the $4$th coordinate.  As $\theta_4$ approaches $2\pi$ we are forced to cross the wall $W_{\{1,4\}}$ into $\cb_3$. So the piecewise polynomial becomes $\mathrm{Vol}\big(\cM^{\text{hyp}}_{g,n}(i\boldsymbol{\theta})\big)=V_{0,\cb_3}(i\b{\theta})$, and:
\[
     \left.\frac{\partial}{\partial \theta_{4}}\right|_{\theta_4=2\pi} \hspace{-5mm}\mathrm{Vol}\big(\cM^{\text{hyp}}_{g,n}(i\theta_1,i\theta_2,i\theta_3,i\theta_4)\big)=-2\pi (\theta_1+\theta_2+\theta_3-2),
\]
whereas
\[
    \left.\frac{\partial}{\partial \theta_{4}}\right|_{\theta_4=2\pi} \hspace{-5mm}V_{0,\cb_2}(i\theta_1,i\theta_2,i\theta_3,i\theta_4)=-2\pi (\theta_2+\theta_3-1).
\]
The first calculation is proven in \cite[(1.3b)]{ETu2Dd}.   The second calculation requires 
Theorem~\ref{Thm:Dilaton}.

%%%%%%%%%%%%%%%%%%%%%%%%%%%%%%%%%%%%%%%%%%%%%%%%%%%%%%%%
%%%%%%%%%%%%%%%%%%%%%%%%%%%%%%%%%%%%%%%%%%%%%%%%%%%%%%%%
\section{Volume and wall-crossing calculations}  \label{sec:ex}
In this section, we calculate the polynomials $V_{0,\cc_1}(i\b{\theta})$ for the minimal chamber $\cc_1$ which are used in the wall-crossing polynomial in Theorem~\ref{wall-crossing formula} which leads to a more explicit formula \eqref{wckernel}.  We also consider other genus 0 examples, in particular, the three genus $0$ examples discussed in \cite{HasMod}, where the moduli spaces $\overline{\cM}_{0,\cA}$ are isomorphic to Losev-Manin space \cite{LMaNew}, $(\bbC P^1)^{n}$ and $\bbC P^{n}$. We also separately look at some small $n$ examples. %The $\bbC P^{n}$ computation gives us an explicit formula for $V_{1,\cc_1}$ volume polynomials, which arise in Theorem~\ref{wall-crossing formula}. %The implication of this on wall-crossing is expressed in Theorem~\ref{thm:wall2}. 
Theorems~\ref{wall-crossing formula} and~\ref{Thm:Dilaton} and Lemma~\ref{Thm:2pilimits} are crucial for the computations in this section.

%For 
%\begin{align*}
%a_1+a_2+a_3&\geq 2,\\
 %a_i+a_j&>1\qquad\{i,j\}\subset\{1,2,3\}\\
 %\sum_{i\in I}a_i&<1\qquad|I^c\cap\{1,2,3\}|\geq 2
 %\end{align*}
%\[ V_{0,n}^{[\vec{a}]}(b_1,...,b_n)=b_4...b_n(\sum_{i=1}^nb_i -2)^n
%\]

%\addtocontents{toc}{\protect\setcounter{tocdepth}{0}}

\subsection{Minimal chamber} \label{volP^n}

The minimal chamber $\cc_1$ defined in \eqref{ming=0} is given by
 \begin{equation*}
    \cc_1:\cP(\b{n})\rightarrow \{0,1\},\qquad \cc_1(J)=0\quad\Leftrightarrow\quad 1\not\in J\text{ and } J^c\neq \{1\}.
\end{equation*}
This is the unique chamber of $\cD_{0,n}$ containing the element $\b{b}_n=(1,b^{n-1})$, for $\frac{1}{n-1}<b\leq\frac{1}{n-2}$. The moduli space corresponding to this chamber is isomorphic to $\bbC P^{n-3}$ \cite[\textsection 6.2]{HasMod}, proven via a series of blow-downs.
\begin{proposition}\label{lem:lightchamber}
 In the chamber $\cc_1\subset\cD_{0,n}$, \begin{equation}
    V_{0,\cc_1}(i\b{\theta})=\frac{(2\pi^2)^{n-3}}{(n-3)!} \left(-2+\sum_{j=1}^{n} a_j\right)^{n-3}\left(-a_1+\sum_{j=2}^{n} a_j\right)^{n-3}, \label{Vcc0}
\end{equation}
where $a_j=a(\theta_j)$.
\end{proposition}
\begin{proof}[Proof 1]
We prove this volume formula using Theorems~\ref{wall-crossing formula} and~\ref{Thm:Dilaton}. The $\cc_1$ chamber has no flat coordinates, so Theorem~\ref{Thm:Dilaton} cannot be applied directly. We first use Theorem~\ref{wall-crossing formula} to express $ V_{0,\cc_1}(i\b{\theta})$ as a sum of wall-crossing polynomials and a volume polynomial corresponding to a chamber where the $n$th coordinate is flat.

Define the chamber $\cb$ to be above $\cc_1$ and incident to the wall $W_{S}$, for $S=\{2,\ldots,n-1\}$.
Theorem~\ref{wall-crossing formula} gives an expression for the wall-crossing polynomial:
\begin{align*}
    \wc_{\cb,S}(\boldsymbol{\theta})&=\int_{0}^{\phi_{S}} 1 \cdot V_{0,\cc_1}(i\theta,i\theta_2,\ldots,i\theta_{n-1})  \cdot \theta d\theta,
\end{align*}
where $\boldsymbol{\theta}=(\theta_1,\ldots,\theta_n)$ and $\phi_{S}=\sum_{j=2}^{n-1}\theta_j-2\pi (n-3)$ is independent of $\theta_n$. 

The $n$th point in $\cb$ is flat, so Theorem~\ref{Thm:Dilaton} applies to $V_{0,\cb}$ with $q(\cb)=\{1\}$. Differentiate with respect to $\theta_n$ and evaluate at $2\pi$ to produce, 
\begin{align}
    \left.\frac{\partial}{\partial \theta_{n}}\right|_{\theta_n=2\pi}  \hspace{-4mm}V_{0,\cc_1}(i\b{\theta})
    &=\left.\frac{\partial}{\partial \theta_{n}}\right|_{\theta_n=2\pi} \Big( V_{0,\cb}(i\b{\theta})+\int_{0}^{\phi_{S}} V_{0,\cc_1}(i\theta,i\theta_2,\ldots,i\theta_{n-1})  \cdot \theta d\theta\Big) \nonumber \\
    &\ =-2\pi (-1+\sum_{j=2}^{n-1}a_j)V_{0,\cb|_{\b{n-1}}}(i\theta_1,\ldots,i\theta_{n-1})+0, \label{dilcpn}
\end{align}
where $\cb|_{\b{n-1}}=\cc_1\subset\cD_{0,n-1}$ is a minimal chamber, also denoted $\cc_1$, obtained by forgetting the $n$th point. So \eqref{dilcpn} becomes
\begin{equation} \label{cc0rec}
    \left.\frac{\partial}{\partial \theta_{n}}\right|_{\theta_n=2\pi} \hspace{-4mm}V_{0,\cc_1}(i\b{\theta})=-2\pi \Big(-1+\sum_{j=2}^{n-1}a_j\Big)V_{0,\cc_1}(i\theta_1,\ldots,i\theta_{n-1}).
\end{equation}
It can be show that the formula in \eqref{Vcc0} satisfies this recursive equation. Give this formula the name $\cV_{n}$, so
\begin{equation*}
    \cV_{n}(\b{\theta})=\frac{(2\pi^2)^{n-3}}{(n-3)!} \left(-2+\sum_{j=1}^n a_j\right)^{n-2}\left(-a_1+\sum_{j=2}^n a_j\right)^{n-3}.
\end{equation*}
To prove uniqueness of $\cV_{n}$ in satisfying \eqref{cc0rec}, we follow an inductive argument on $n$. In the case $n=3$, we know from Section~\ref{ex:M04} that equation \eqref{Vcc0} holds. Assume equation \eqref{Vcc0} holds for $n<k$. Suppose if other than $\cV_{k}$ that there is another degree $2k-6$ polynomial $f:[0,1)^{k}\rightarrow\bbR$, symmetric in $\theta_j$ for $j\in\{2,\ldots,k\}$, satisfying \eqref{Vcc0}. Then 
\begin{equation}
     \left.\frac{\partial}{\partial \theta_{k}}\right|_{\theta_k=2\pi}\hspace{-4mm}(\cV_{k}-f)(\theta_1,\ldots,\theta_k)=0. \label{difference0}
\end{equation}
Equation \eqref{difference0} and symmetry mean that the difference must satisfy 
\begin{equation*}
    (\cV_{k}-f)(\theta_1,\ldots,\theta_k)=(2\pi-\theta_2)^2\cdots (2\pi-\theta_k)^2\cdot g(\theta_1,\ldots,\theta_k),
\end{equation*}
for some polynomial $g$. However, if $g$ were non-zero then the degree of $\cV_{k}-f$ would exceed $2k-6$. Hence, it must be that $g\equiv 0$ and $\cV_{k}\equiv f$. This proves uniqueness of the set of functions $\{\cV_{k}\}_{k\geq 3}$ that satisfies the recursion \eqref{cc0rec}. So $V_{0,\cc_1}(i\b{\theta})=\cV_{n}(\b{\theta})$.
\end{proof}
\begin{proof}[Proof 2]  This proof exploits the simplicity of $H^2(\cc_1,\bz)$.
Since $H^2(\bbC P^{n-3},\bbZ)\cong\bbZ$, all divisors on $\overline{\cM}_{0,\cc_1}$ are proportional and the intersection numbers of any non-trivial curve with two divisors determines this proportionality. Let
$$\pi:\Mbar{0}{n}\longrightarrow \Mbar{0}{\cc_1}$$
be the reduction morphism from the Deligne-Mumford compactification. Denote by $B_j$ the curve obtained in $\Mbar{0}{n}$ by fixing all points in $\b{n}\setminus\{j\}$ in general position on a rational curve and allowing the $j$th point to vary. The top intersection numbers with the $\psi$ and boundary classes of this curve are
$$B_j\cdot\psi_i=\begin{cases}n-3&\text{for $i=j$}\\
1&\text{otherwise,} \end{cases}\quad\text{   and   }\quad B_j\cdot \delta_{0:S}=\begin{cases} 1&\text{if $|S|=2$ and $j\in S$}\\1&\text{if $|S^c|=2$ and $j\in S^c$}\\0&\text{otherwise,} \end{cases}$$
and Proposition~\ref{prop:kappaclass} gives $B_j\cdot\kappa_1=n-3$. Hence via \eqref{redmor}, we obtain top intersections 
$$B_1\cdot\pi^*\psi_i=\begin{cases}n-3&\text{for $i=1$}\\
3-n&\text{otherwise,} \end{cases}\quad\text{   and   }\quad B_1\cdot \pi^*D_{i,j}=n-3.$$
Further, Lemma~\ref{lem:kappapullback1} gives
$$B_1\cdot(\pi^*\kappa_1)=B_1\cdot\left(\kappa_1+\sum_{S\in\textbf{red}(\pi)}(|S|-1)^2\delta_{0:S}\right)=(n-3)(n-2)^2$$
where $\textbf{red}(\pi)$ denotes the reductions sets of $\pi$ and the intersection number follows by observing that when $S\subset\{2,\dots,n\}$ then $B_1\cdot\delta_{0:S}=1$ for $|S|=n-2$ and zero otherwise.

Similarly, for $k\in\{2,\dots,n\}$ we obtain
$$B_k\cdot\pi^*\psi_i=\begin{cases}1&\text{for $i=1$}\\
-1&\text{otherwise,} \end{cases}\quad B_k\cdot \pi^*D_{i,j}=1,\quad \text{and}\quad B_k\cdot\pi^*\kappa_1=(n-2)^2$$

Hence on $ \Mbar{0}{\cc_1}$ 
$$\psi_1=D_{i,j}=\frac{1}{(n-2)^2}\kappa_1=-\psi_k\quad\text{ for $k=2,\dots,n$}$$
and we obtain
\begin{eqnarray*}
\gamma(\b{a})&=&\left((n-2)^2-(1-a_1)^2+\sum_{j=2}^{n}(1-a_j)^2+2\sum_{2\leq i<j\leq n}(a_ia_j-1)\right)\psi_1\\
&=& \left((1-\sum_{j=2}^{n}a_j)^2-(1-a_1)^2\right)\psi_1.
\end{eqnarray*}
Hence 
 \begin{eqnarray*}
 V_{0,\cc_1}(i\b{\theta}) &=&\frac{(2\pi^2)^{n-3}}{(n-3)!} \int_{\overline{\cM}_{0,\cc_1}}\gamma(\b{a})^{n-3}\\
 &=&\frac{(2\pi^2)^{n-3}}{(n-3)!} \left((1-\sum_{j=2}^{n}a_j)^2-(1-a_1)^2\right)^{n-3} \int_{\overline{\cM}_{0,\cc_1}}\psi_1^{n-3}\\
 &=&\frac{(2\pi^2)^{n-3}}{(n-3)!} \left((1-\sum_{j=2}^{n}a_j)^2-(1-a_1)^2\right)^{n-3} 
\end{eqnarray*}
which agrees with \eqref{Vcc0}.
\end{proof}
From this calculation, the wall-crossing polynomial \eqref{WCF} becomes
\begin{equation}  \label{wckernel}
\wc_{\cc,S}(\boldsymbol{\theta})=\frac{1}{(|S|-2)!\cdot 2^{|S|-2}} \int_{0}^{\phi_S} V_{g,\cc/S}(i\boldsymbol{\theta}_{S^c},i\theta)\cdot \left(\phi_S^2-\theta^2\right)^{|S|-2} \cdot \theta \cdot d\theta
\end{equation}
where $\phi_S=\sum_{j\in S}\theta_j-2\pi(|S|-1)$.

\begin{corollary} \label{cor:wallderiv}
    For any $j\in S$, %when $|S|>2$ 
    $$   \left.\frac{\partial}{\partial \theta_{j}}\right|_{\theta_j=2\pi} \wc_{\cc,S}(\boldsymbol{\theta})=\begin{cases}\phi_{S\setminus\{j\}}\cdot \wc_{\cc|_{\b{n}\setminus\{j\}},S\setminus\{j\}}(\boldsymbol{\theta})&\text{ if $|S|>2$}\\
 \theta_{S\setminus\{j\}}\cdot V_{g,\cc/S}(i\b{\theta})&\text{ if $|S|=2$}.\end{cases}$$
   % \phi_{S\setminus\{j\}}\cdot \wc_{\cc|_{\b{n}\setminus\{j\}},S\setminus\{j\}}(\boldsymbol{\theta}).$$
\end{corollary}
\begin{proof}
    Let $j\in S$. In the case where $|S|=2$, wall-crossing is given simply by $\wc_{\cc,S}(\boldsymbol{\theta})=\int_{0}^{\phi_S} V_{g,\cc/S}(i\boldsymbol{\theta}_{S^c},i\theta) \cdot \theta \cdot d\theta$, and the integrand is independent of $\theta_j$. The result follows since $\phi_S\mid_{\theta_{j}=2\pi}=\theta_{S\setminus\{j\}}$.
    
    In the case where $|S|>2$, 
    \begin{align*}
        \frac{\partial}{\partial \theta_{j}}\wc_{\cc,S}(\boldsymbol{\theta})&=\frac{1}{(|S|-2)!\cdot 2^{|S|-2}} \int_{0}^{\phi_S} V_{g,\cc/S}(i\boldsymbol{\theta}_{S^c},i\theta)\cdot\theta\cdot
         \frac{\partial}{\partial \phi_S}\left(\phi_S^2-\theta^2\right)^{|S|-2} d\theta \\
         &=\frac{1}{(|S|-3)!\cdot 2^{|S|-3}} \int_{0}^{\phi_S} V_{g,\cc/S}(i\boldsymbol{\theta}_{S^c},i\theta)\cdot\theta\cdot
         [(\phi_S^2-\theta^2)^{|S|-3}\cdot\phi_S] d\theta.
    \end{align*}
    Set $\cc|_j:=\cc|_{\b{n}\setminus\{j\}}$ and $S_j:=S\setminus\{j\}$. Since $\phi_S\mid_{\theta_{j}=2\pi}=\phi_{S\setminus\{j\}}$, it follows that
        \begin{align*}
        &\left.\frac{\partial}{\partial \theta_{j}}\right|_{\theta_j=2\pi}\hspace{-4mm}\wc_{\cc,S}(\boldsymbol{\theta}) \\
        &\qquad =\frac{1}{(|S|-3)!\cdot 2^{|S|-3}} \int_{0}^{\phi_{S_j}} V_{g,\cc'/S_j}(i\boldsymbol{\theta}_{S_j^c},i\theta)\cdot\theta\cdot
         [(\phi_{S_j}^2-\theta^2)^{|S|-3}\cdot\phi_{S_j}] d\theta \\
        &\qquad =\phi_{S_j}\cdot \wc_{\cc|_j,S_j}(\boldsymbol{\theta}).
    \end{align*}
\end{proof}

Theorem~\ref{Thm:Dilaton} applies to chambers that are flat in the $(n+1)$th coordinate.  It can be generalised to other chambers as follows. Let  $\cc\subset\cd_{g,n+1}$ be any chamber and $\{\cb_j,S_j\}_{j=1}^k$ such that $n+1\in S_j$ and $|S_j|\geq 3$ be a series of simple wall-crossings from $\cc=\cb_1$ to the unique chamber $\cb_{k+1}$ that is flat in the $(n+1)$th coordinate. Then
 \begin{align}\label{eqn:derivativegeneral}\begin{split}
 \frac{\partial V_{g,\cc}}{\partial\theta_{n+1}}(i\theta_1,\dots,i\theta_n,2\pi i)\hspace{-0.5mm}&=\hspace{-0.5mm}-2\pi\Big(2g-2+|q(\cc)|+\hspace{-2mm}\sum_{j\notin q(\cc)} a_j\Big)V_{g,\cc|_{\b{n}}}(i\theta_1,\dots,i\theta_n)\\
&\qquad-\sum_{j=1}^k\frac{\partial\wc_{\cb_j,S_j}}{\partial \theta_{n+1}}(\theta_1,\dots,\theta_{n},2\pi )
 \end{split}\end{align}
Corollary~\ref{cor:wallderiv} is then useful to express these derivatives as wall-crossing polynomials themselves. As an application, we present the following corollary that generalises the dilaton equation to a broader class of chambers.

\begin{corollary}\label{cor:generalchamber}
Let $\cc \subset \cD_{g, n+1}$ be any chamber %not flat in the $(n+1)$th coordinate
and $\left\{\cB_j, S_j \cup\{n+1\}\right\}_{j=0}^k$ for $S_j \subset \boldsymbol{n}$ be any series of simple wall-crossings from a chamber $\cB_0$, flat in the $(n+1)$th coordinate, to chamber $\cc$ such that all $|S_j|\geq2$. Then
\begin{align*}
\frac{\partial V_{g, \mathcal{C}}}{\partial \theta_{n+1}}(i\b{\theta},2\pi i)=- &2 \pi\Big(2 g-1+\left|q\left(\cc_0\right)\right|+\sum_{j \notin q(\cc)} a_j-\sum_{i \in S_0} a_i\Big) V_{g,\left.\cB_0\right|_n}(i\b{\theta}) \\
& +\sum_{j=1}^k 2 \pi\Big(\sum_{i \in S_j} a_i-\sum_{i \in S_{j-1}} a_i\Big) V_{g,\left.\cB_j\right|_n}(i\b{\theta}) \\
& +2 \pi\Big(1-\sum_{i \in S_k} a_i\Big) V_{g,\left.\cc\right|_n}(i\b{\theta}),
%(removing the $|S_j|\geq2$ condition include also the term &&+\sum_{|S_j|=1}2\pi(1-a_{S_j})V_{g,\cc_j\_{\b{n}}}(i\theta_1,\dots,i\theta_n)
\end{align*}
where $q(\cc)=\left\{j \mid W_{\{j, n+1\}}\right.$ is crossed from $\cc_0$ as $\left.\theta_{n+1} \rightarrow 2 \pi\right\}$ and $\b{\theta}=(\theta_1,\ldots,\theta_n)$.
\end{corollary}

\begin{proof}
Corollary~\ref{cor:wallderiv} states for $\left|S_j\right| \geq 2$,
$$\frac{\partial \wc_{\cB_j, S_j \cup\{n+1\}}}{\partial \theta_{n+1}}(i \b{\theta}, 2\pi)=2 \pi\Big(1-\sum_{i \in S_j} a_i\Big) \wc_{\left.\cB_j\right|_n, S_j}(i \b{\theta})$$
%and for $S_j=\{t\}$
%$$\frac{\partial \wc_{\cc_j,S_j\cup\{n+1\}}(i\theta_1,\dots,i\theta_{n},2\pi i)}{\partial\theta_{n+1}}=2\pi i(1-a_t)V_{g,\cc_j/\{t,n+1\}}(i\theta_1,\dots,i\theta_{n}).  $$
%However, for $S_j=\{t\}$ we have $\cc_j/\{t,n+1\}=\cc_j|_{\b{n}}=\cc_{j-1}|_{\b{n}}$. 
The result then follows by differentiating and evaluating both sides of the equation
$$V_{g, \cc}(i\b{\theta},i\theta_{n+1})=V_{g, \cB_0}(i\b{\theta},i\theta_{n+1})+\sum_{j=0}^k \wc_{\cB_j, S_j \cup\{n+1\}}(i\b{\theta},i\theta_{n+1})$$
and eliminating the wall-crossing terms in the result.
\end{proof}
A similar formula can be derived allowing for the possibility of walls $W_{S_j\cup\{n+1\}}$ with $|S_j|=1$. Such a formula generalises the dilaton equation to all chambers.

\subsection{Losev-Manin Spaces}
Define the \textit{Losev-Manin chamber} \begin{equation*}
    \cL_n:\cP(\b{n+2})\rightarrow \{0,1\},\qquad \cL_n(J)=0\quad\Leftrightarrow\quad \{n+1,n+2\}\cap J=\emptyset.
\end{equation*}
It is the unique chamber of $\cD_{0,n+2}$ which contains the element $\b{l}_n:=(\overbrace{\epsilon,\ldots,\epsilon}^n,1,1)$ for $\epsilon>0$ small.

The moduli space $\overline{\cM}_{0,\b{l}_n}$ defined in the chamber $\cL_n$ was studied by Losev and Manin in \cite{LMaNew}, and later by Hassett \cite{HasMod} and Cavalieri \cite{CavMod}. As such, $\overline{\cM}_{0,\b{l}_n}$ is called the \textit{Losev-Manin space}. \\ \indent 
In this section we use intersection theory on $\overline{\cM}_{0,\b{l}_n}$ in order to compute the polynomial
$V_{0,\cL_n}(i\b{\theta})$, in the case where $\theta_{n+1}\equiv\theta_{n+2}\equiv 0$. This corresponds to a volume polynomial for
$\cM^{\text{hyp}}_{0,n+2}(i\b{\theta})$, where $\b{a}(\b{\theta})=(\epsilon_1,\ldots,\epsilon_n,1,1)\in\cL_n$. 
\begin{lemma}
    In the Losev-Manin chamber $\cL_n$, \begin{equation*}
    V_{0,\cL_n}(i\b{\theta},0,0)=(2\pi)^{2(n-1)}\left(\prod_{j=1}^n\epsilon_j\right)\left(\sum_{k=1}^n\epsilon_k\right)^{n-2},
\end{equation*}
   where $\epsilon_k=a(\theta_k)$ for $k\in\{1,\ldots,n\}$ and $\theta_{n+1},\theta_{n+2}$ are fixed as constant in the ring of functions, such that $a(\theta_{n+1}),a(\theta_{n+2})\equiv 1$.\\ \indent 
  Moreover, under the identification $\epsilon_i=\epsilon_j=\epsilon$ there is a recursion formula, \begin{equation*}
 \cV_n(\epsilon)=(2\pi \epsilon)^2\sum_{\substack{i+j=n\\ 1\leq i\leq j\leq n-1}}\frac{ij}{n-1}{n \choose i}\cV_i(\epsilon)\cV_j(\epsilon),
 \end{equation*}
 where $\cV_k(\epsilon)= V_{0,\cL_k}(i\b{\theta})$ in this identification ie. such that $\b{a}(\b{\theta})=\b{l}_n$.
\end{lemma}

\begin{proof}
Define the sets $S_j=\{n,n+j\}$ for $j=1$ or $2$.  Cross the walls $W_{S_1}$ and $W_{S_2}$ to travel from 
$\cL_n$ to a chamber $\cB$ which is light in the $n$th coordinate, where Lemma~\ref{Thm:2pilimits} applies.
Using \eqref{eqn:incidentzerowc}, we have \begin{equation} %\label{dil1LM}
    V_{0,\cL_n}(i\b{\theta})\mid_{\epsilon_n=0}=-\sum_{j=1}^2\wc_{\cb_j,S_j}(\theta_1,\dots,\theta_{n-1},2\pi ,\theta_{n+1},\theta_{n+1}),
\end{equation}
where $\cB_1=\cc$ and $\cB_2$ is intermediate chambers in the wall-crossing. Evaluation at 
$\theta_{n+1}=\theta_{n+2}=0$ of these wall-crossing polynomials causes them to vanish, since 
$\phi_{S_j}=\theta_n+\theta_j-2\pi\mid_{\theta_n=2\pi}\to 0$, so by Theorem~\ref{wall-crossing formula}, \begin{equation} \label{dil1LM}
    V_{0,\cL_n}(i\b{\theta},0,0)\mid_{\epsilon_n=0}=0. 
\end{equation}
The Losev-Manin chamber $\cL_n$  is flat in the $n$th coordinate, this means we can apply Theorem~\ref{Thm:Dilaton} with $q(\cc)=\{n+1,n+2\}$.
Using \eqref{eqn:derivative}, we have \begin{align}
   \left.\frac{\partial}{\partial \epsilon_{n}}\right|_{\theta_n=2\pi} \hspace{-4mm}V_{0,\cL_n}(i\b{\theta},0,0)&= 4\pi^2\Big(2g-2+2+\sum_{j=1}^{n-1} 
   a(\theta_j)\Big)V_{0,\cL_{n-1}}(i\theta_1,\ldots,i\theta_{n-1},0,0)  \nonumber \\
    &=4\pi^2\left(\sum_{j=1}^{n-1}a(\theta_j)\right)V_{0,\cL_{n-1}}(i\theta_1,\ldots,i\theta_{n-1},0,0),\label{dil2LM}
\end{align}
where we have used that $\cL_n|_{\b{n+2}\setminus\{n\}}=\cL_{n-1}$. 

We know that $V_{0,\cL_n}(i\b{\theta},0,0)$ is a $2(n-1)$-degree symmetric polynomial in $\epsilon_j$. Additionally, \eqref{dil1LM} implies that it must have a factor of $\epsilon_1\epsilon_2\cdots \epsilon_n$. Define $(n-2)$-degree symmetric polynomials $P_n:[0,1]^n\rightarrow \bbR$ such that $$V_{0,\cL_n}(i\b{\theta},0,0)=\epsilon_1\epsilon_2\cdots \epsilon_nP_n(\epsilon_1,\epsilon_2,\ldots, \epsilon_n).$$ Then \eqref{dil2LM} gives the identity \begin{equation}
    P_n(\epsilon_1,\ldots,\epsilon_{n-1},0)=4\pi^2\left(
    %[[n-1-2^{n-1}]]+
    \sum_{k=1}^{n-1}\epsilon_k\right)P_{n-1}(\epsilon_1,\ldots,\epsilon_{n-1}). \label{recursion1}
\end{equation}
Since $P_n$ is symmetric, its evaluation at $\epsilon_n=0$ defines it uniquely. The result from Section~\ref{ex:M04} in chamber $\cB_{1}$ (which is equivalent to the chamber $\cL_2$ in this section) gives that $V_{0,\cL_2}(i\theta_1,i\theta_2,0,0)=4\pi^2\epsilon_1\epsilon_2$. It follows that $P_2(\epsilon_1,\epsilon_2)=4\pi^2$. Starting from here, and using \eqref{recursion1}, we have that
 \begin{equation*}
    P_n(\epsilon_1,\ldots,\epsilon_n)=(2\pi)^{2(n-1)}\left(\sum_{k=1}^n\epsilon_k\right)^{n-2}, 
\end{equation*}
so \begin{equation*}
   V_{0,\cL_n}(i\b{\theta},0,0)=(2\pi)^{2(n-1)}\left(\prod_{j=1}^n\epsilon_j\right)\left(\sum_{k=1}^n\epsilon_k\right)^{n-2}.
\end{equation*}
Let $\cV_k(\epsilon)= V_{0,\cL_k}(i\b{\theta},0,0)\vert_{\epsilon_i=\epsilon_j}$ for $k\geq 2$. Then $\cV_{n+2}(\epsilon)=(2\pi)^{2(n-1)}n^{n-2}\epsilon^{2(n-1)}$. To write this as a recursive formula, we use Cayley's formula for counting the number of different labelled trees on $n$ vertices \cite{CayTre, ShuSho}. Using this identity, we obtain 
\begin{align*}
    \cV_{n+2}(\epsilon)&=\sum_{\substack{i+j=n\\ 1\leq i\leq j\leq n-1}}n\cdot(2\pi \epsilon)^2{n-2 \choose i-1}\cV_{i+2}(\epsilon)\cV_{j+2}(\epsilon)\\
    &=(2\pi \epsilon)^2\sum_{\substack{i+j=n\\ 1\leq i\leq j\leq n-1}}\frac{ij}{n-1}{n \choose i}\cV_{i+2}(\epsilon)\cV_{j+2}(\epsilon).
\end{align*}
%In particular, setting $\epsilon=\frac{1}{n}$ gives a volume recursion for the prototypical Losev-Manin space in \cite{LMaNew}.
\end{proof}

\subsection{Chamber isomorphic to $(\bbC P^1)^n$}
Define the chamber  \begin{equation*}
    \cA_n:\cP(\b{n+3})\rightarrow \{0,1\},\qquad \cA_n(J)=0\quad\Leftrightarrow\quad |\{n+1,n+2,n+3\}\cap J|\leq 1.
\end{equation*}
It is the unique chamber of $\cD_{0,n+3}$ which contains the element 
\[\b{a}_n=(\overbrace{\epsilon,\ldots,\epsilon}^n,2/3,2/3,2/3)\] for $\epsilon>0$ small. \\ \indent 
The moduli space $\overline{\cM}_{0,\b{a}_n}$ defined in the chamber $\cA_n$ is isomorphic to $(\bbC P^1)^{n}$. This compactification of $\cM_{0,n}$ was studied in \cite[\textsection 6.3,\,8]{HasMod}. \\ \indent
In this section we use intersection theory on $\overline{\cM}_{0,\b{a}_n}$ in order to compute the polynomial
$V_{0,\cA_n}(i\b{\theta})$. This corresponds to a volume polynomial for
$\cM^{\text{hyp}}_{0,n+3}(i\b{\theta})$, where $\b{a}(\b{\theta})=(\epsilon_1,\ldots,\epsilon_n,b_1,b_2,b_3)\in\cA_n$. 
\begin{lemma}
   In the chamber $\cA_n$, \begin{equation*}
    V_{0,\cA_n}(i\b{\theta})=(2\pi)^{2n}\left(\prod_{j=1}^n\epsilon_j\right)\left(-2+\sum_{j=1}^n\epsilon_j+\sum_{k=1}^3 b_k\right)^n,
\end{equation*}
where $\epsilon_j=a(\theta_j)$ for $j\in\{1,\ldots,n\}$ and $b_k=a(\theta_{n+k})$ for $k\in\{1,2,3\}$.
\end{lemma}
\begin{proof}
The first $n$ points corresponding to the $\cA_n$ chamber are light, so here Lemma~\ref{Thm:2pilimits} holds. Using \eqref{eqn:incidentzero}, we have \begin{equation} \label{dil3LM}
    V_{0,\cA_n}(i\b{\theta})\mid_{\epsilon_j=0}=0.
\end{equation}
Lightness of the first $n$ points means we may also apply Theorem~\ref{Thm:Dilaton} with $q(\cc)=\emptyset$.
Using \eqref{eqn:derivative}, we have \begin{align} \begin{split}
   &\left.\frac{\partial}{\partial \epsilon_{n}}\right|_{\theta_n=2\pi} \hspace{-4mm} V_{0,\cA_n}(i\b{\theta}) \\
   &= 4\pi^2\Big(2g-2+\sum_{j=1}^{n+3}a(\theta_j)\Big)\vert_{a(\theta_n)=0}V_{0,\cA_{n-1}}(i\theta_1,\ldots,i\theta_{n-1},i\theta_{n+1},i\theta_{n+2},i\theta_{n+3})   \\
    &=4\pi^2\left(-2+\sum_{j=1}^{n-1}\epsilon_j+\sum_{k=1}^{3}b_k\right)V_{0,\cA_{n-1}}(i\theta_1,\ldots,i\theta_{n-1},i\theta_{n+1},i\theta_{n+2},i\theta_{n+3}),\label{dil4LM} \end{split} 
\end{align}
where we have used that $\cA_n|_{\b{n+3}\setminus\{n\}}=\cA_{n-1}$.

 We know that $V_{0,\cA_n}(i\b{\theta})$ is a $2n$-degree polynomial, which is symmetric in $\epsilon_j$ and symmetric $b_k$. Additionally, \eqref{dil3LM} implies that it must have a factor of $\epsilon_1\epsilon_2\cdots \epsilon_n$. Define $n$-degree polynomials $Q_n:[0,1]^{n+3}\rightarrow \bbR$ which are symmetric in the first $n$ entries and final $3$ entries, such that $$V_{0,\cA_1}(i\b{\theta})=\epsilon_1\epsilon_2\cdots \epsilon_n Q_n(\epsilon_1,\epsilon_2,\ldots, \epsilon_n,b_1,b_2,b_3).$$ Then \eqref{dil4LM} gives the identity \begin{align} \begin{split}
    Q_n(\epsilon_1,\ldots,\epsilon_{n-1}&,0,b_1,b_2,b_3) \\
    &=4\pi^2\left(-2+\sum_{j=1}^{n-1}\epsilon_j+\sum_{k=1}^{3}b_k\right)Q_{n-1}(\epsilon_1,\ldots,\epsilon_{n-1},b_1,b_2,b_3). \end{split} \label{recursion2}
\end{align}
Since $Q_n$ is symmetric in its first $n$ entries, evaluation at $\epsilon_n=0$ defines it uniquely. 
The result from Section~\ref{ex:M04} in chamber $\cB_{3}$ (which is equivalent to the chamber $\cA_1$ in this section) gives that $V_{0,\cA_n}(i\b{\theta})=4\pi^2\epsilon_1(-2+\epsilon_1+b_1+b_2+b_3)$. It follows that $Q_1(\epsilon_1,b_1,b_2,b_3)=4\pi^2$. Starting from here, and using \eqref{recursion2}, we have that \begin{equation*}
    Q_n(\epsilon_1,\ldots,\epsilon_n,b_1,b_2,b_3)=(2\pi)^{2n}\left(-2+\sum_{j=1}^n\epsilon_j+\sum_{k=1}^nb_j\right)^{n}, 
\end{equation*}
so \begin{equation*}
   V_{0,\cA_n}(i\b{\theta})=(2\pi)^{2n}\left(\prod_{j=1}^n\epsilon_j\right)\left(-2+\sum_{j=1}^n\epsilon_j+\sum_{k=1}^3 b_k\right)^n
\end{equation*}
\end{proof}

%%%%%%%%%%%%%%%%%%%%%%%%%%%%%%%%
\subsection{$\overline{\modm}_{1,\bold{a}}$ for $n=2$}
In this section, we compute the volume polynomials of $\overline{\modm}_{1,\bold{a}}$ for $n=2$ in each chamber. There are two chambers to consider, $\cb_j:\cP(\b{2})\rightarrow\{0,1\}$, where $\cb_0=\cc^M$ and $\cb_1=\cc^L$ correspond to the maximal and light chamber respectively. These chambers are separated by a wall $W_{\{1,2\}}$. \\ \indent 
In $\cb_0$, Mirzakhani's volume recursion can be used to compute \begin{align*}
    V_{1,\cb_0}(i\b{\theta})=V_{1,2}^{\text{Mirz}}(i\b{\theta})&=\frac{1}{192}(4\pi^2-\theta_1^2-\theta_2^2)(12\pi^2-\theta_1^2-\theta_2^2). 
%    \\
%    &=\frac{1}{192}(1-(1-a_1)^2-(1-a_2)^2)(3-(1-a_1)^2-(1-a_2)^2),
\end{align*}
%where $a_j=a(\theta_j)$. \\ \indent 
Apply Theorem~\ref{g,2 volume} to compute the single wall-crossing polynomial 
\begin{equation*}
    \wc_{\cc,\{1,2\}}(\boldsymbol{\theta})=\int_{0}^{\phi_{\{1,2\}}}  V_{1,1}(i\theta)\cdot \theta  d\theta.
\end{equation*} 
The integrand is the Mirzakhani volume polynomial, $ V_{1,1}(i\theta)=\frac{1}{48}(4\pi^2-\theta^2)$. So \begin{align*}
      \wc_{\cc,\{1,2\}}(\boldsymbol{\theta})&=\int_{0}^{\theta_1+\theta_2-2\pi}  \frac{1}{48}(4\pi^2\theta-\theta^3) d\theta \\
      &=\frac{1}{192}(\theta_1+\theta_2-2\pi)^2(8\pi^2-(\theta_1+\theta_2-2\pi)^2).
\end{align*}
Hence \begin{align*}
    V_{1,\cb_1}(i\b{\theta})&= V_{1,\cb_0}(i\b{\theta})+ \wc_{\cc,\{1,2\}}(\boldsymbol{\theta}) \\
    &=\frac{1}{48}(2\pi-\theta_1)(2\pi-\theta_2)(8\pi^2-\theta_1^2-\theta_2^2-(2\pi-\theta_1)(2\pi-\theta_2)).
\end{align*}
%Indeed, here one could have applied the more specialised Theorem~\ref{g,2 volume} to compute $V_{1,\cb_1}(i\b{\theta})$ directly. 

%%%%%%%%%%%%%%%%%%%%%%%%%%%
\subsection{$\overline{\modm}_{0,\bold{a}}$ for $n=4$}\label{ex:M04}
In this section we compute the volume polynomials by direct computation to emphasise the geometry.  This is also useful to observe the wall-crossing polynomials in a simple explicit case.  These volumes were also calculated in \cite{ETu2Dd}.
%More so than simply verification, this section should demonstrate the effectiveness of Theorem~\ref{wall-crossing formula} for calculating volume polynomials without the need for long and repetitive computations required when using first principals.  \\ \indent 

Hassett's compactifications of $\modm_{0,4}$ are all isomorphic $\overline{\modm}_{0,\bold{a}}\cong\overline{\modm}_{0,4}\cong \bp^1$ but the universal curve is dependent on the chamber that $\bold{a}$ lies in, which affects \eqref{intfibre} and  hence the volume polynomials.  There are $5$ possibilities, up to symmetry, for chambers 
%$\cb_j:\cP(\b{4})\rightarrow\{0,1\}$, 
%\begin{align*}
%    \cb _0(J)=0\quad&\Leftrightarrow\quad |J|\leq 1, \\
%    \cb _1(J)=0\quad&\Leftrightarrow\quad |J|\leq 1 \text{ or } J=\{3,4\}, \\
%    \cb _2(J)=0\quad&\Leftrightarrow\quad |J|\leq 1 \text{ or } J=\{j,4\} \text{ for } j=2,3, \\
%    \cb _3(J)=0\quad&\Leftrightarrow\quad |J|\leq 1 \text{ or } J=\{j,4\} \text{ for } j=2,3,4, \\
%    \cb _4(J)=0\quad&\Leftrightarrow\quad |J|\leq 1 \text{ or } J=\{j,k\} \text{ for } \{j,k\}\subset\{2,3,4\}. \\
%\end{align*}
$\cb_j\subset\cd_{0,4}$, in terms of cone angles as follows:
\begin{eqnarray*} 
&\cb_0=\cC^M&\hspace{-0.08cm} \ : \ \theta_i+\theta_j<2\pi,\text{ for all $i,j\in\{1,2,3,4\}$},\\[0.8em]
&\cb_1&\hspace{-0.08cm} \ : \  \theta_j+\theta_k<2\pi \text{ for } (j,k)\neq (3,4) \text{ and } \theta_3+\theta_4>2\pi,\\[0.8em]
&\cb_2&\hspace{-0.08cm} \ : \  \theta_j+\theta_k<2\pi \text{ for } \{j,k\}\subset\{1,2,3\},\, 
                               \theta_1+\theta_4<2\pi \\
                               &&\quad\text{ and } \theta_j+\theta_4>2\pi \text{ for } j=2,3,\\[0.8em]
&\cb_3& \hspace{-0.08cm}\ : \ \theta_j+\theta_k<2\pi \text{ for } \{j,k\}\subset\{1,2,3\} \text{ and } \theta_j+\theta_4>2\pi \text{ for } j=1,2,3, \\[0.8em]
&\cb_4=\cc_1& \hspace{-0.08cm} \ : \ \theta_1+\theta_i<2\pi \text{ for } i\in\{2,3,4\} \text{ and }  \theta_j+\theta_k>2\pi \text{ for } \{j,k\}\subset\{2,3,4\}.
\end{eqnarray*}

The chamber $\cb_0$ corresponds to the Deligne-Mumford compactification and the universal cure is obtained as 
$$\pi: \mathcal{X}_0%\cong \overline{\modm}_{0,5}
\longrightarrow S\cong\bp^1$$
where $\mathcal{X}_0=\text{Bl}_\Sigma(\bp^1\times\bp^1)$ is the blow up of $\bp^1\times \bp^1$ at three distinct points in the diagonal $\Sigma=\{(p_1,p_1),(p_2,p_2),(p_3,p_3)\}$ and $\pi$ then is given by the blow down and projection to the second component $\bp^1$.
%\begin{figure}[ht]  \label{conehex}
%	\centerline{\includegraphics[height=6cm]{HassettCom.pdf}}
%	\caption{Hassett's compactification of $\cM_{0,4}$ in the chambers $\cb_0$ and $\cb_1$}
%	\label{fig:hascom}
%\end{figure} 
\begin{figure}[ht]  
	\centerline{\includegraphics[height=4.45cm]{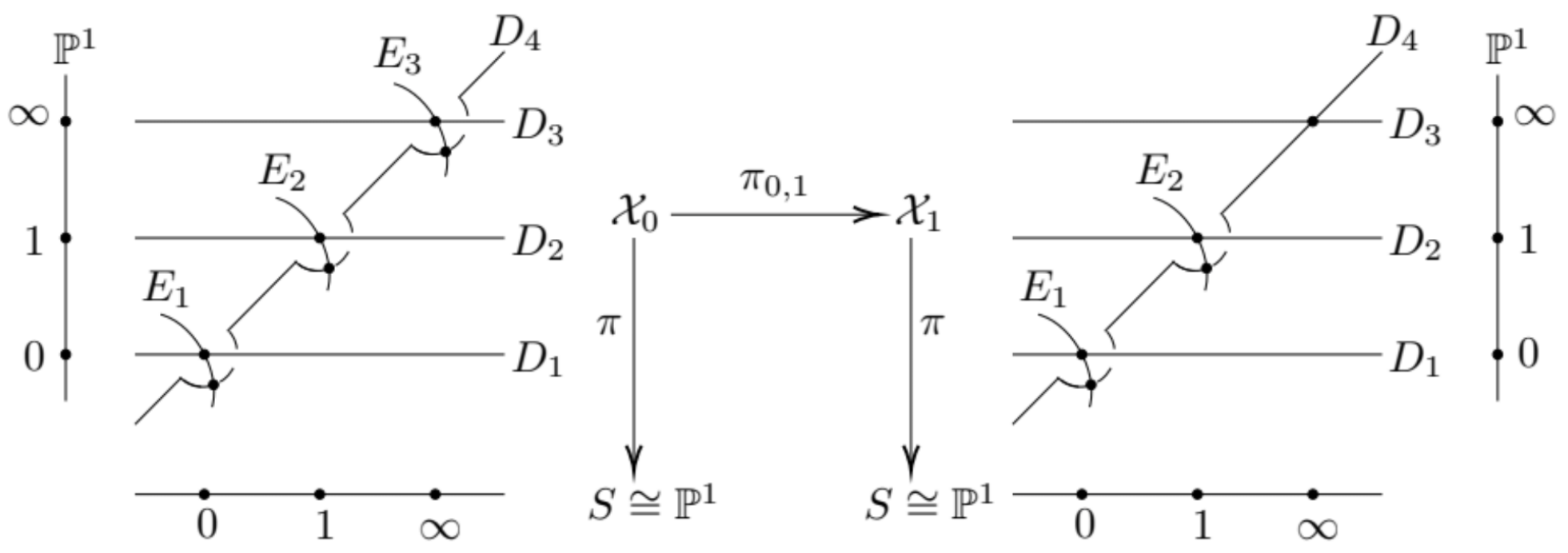}}
	\captionsetup{skip=0.6cm}
\caption{Hassett's compactification of $\cM_{0,4}$ in the chambers $\cb_0$ and $\cb_1$. Here $\pi_{0,1}$ is the birational reduction morphism, which contracts the exceptional divisor $E_3$ to a point.}
	\label{fig:hascom}
\end{figure}

This is shown in the left of Figure~\ref{fig:hascom}. We denote the three exceptional divisors as $E_1,E_2,E_3$ respectively and divisors $D_1$, $D_2$, $D_3$ and $D_4$ as the sections obtained as the proper transforms of the curves $\{p_i\}\times\bp^1$ for $i=1,2,3$ and the diagonal on $\bp^1\times\bp^1$ respectively. Hence if we denote by $H$ and $F$ the pullbacks to $\mathcal{X}_0$ from $\bp^1\times\bp_1$ of the horizontal and vertical fibre classes respectively, we obtain on the level of the classes that $D_j+E_j=H$ for $j=1,2,3$ and $D_4+E_1+E_2+E_3=H+F$. With this in hand we can compute the volume polynomial in the chamber $\cb_0$.
\smallskip

  {\bf Chamber $\cb_0$}: The first Chern class of the relative canonical sheaf is:
\[\hspace{-0.1cm}\frac{1}{2\pi}[\Omega_{\mathcal{X}_0/S}]\hspace{-0.1cm}=\hspace{-0.1cm}K_{\mathcal{X}_0}\otimes \pi^*(K_{S}^{*})\hspace{-0.1cm}=
%\hspace{-0.1cm}-2H-2F+E_1+E_2+E_3+2F\hspace{-0.1cm}=
\hspace{-0.1cm}-2H+E_1+E_2+E_3,
\]
hence the volume polynomial becomes
\begin{eqnarray}
V_{0,\cb_0}(i\b{\theta})&=&\frac12([\Omega_{\mathcal{X}_0/S}]+2\pi\mathbf{a}\cdot D)^2\nonumber\\
&=&\frac 1 2(2\pi (-2H+E_1+E_2+E_3+\sum_{j=1}^4 a_jD_j))^2\nonumber\\
&=&2\pi^2-\frac12\sum_{j=1}^4\theta_j^2, \label{C0vol}
\end{eqnarray}
as $E_i\cdot E_j=0$ for $i\ne j$, $E_i^2=-1$, $H\cdot F=1$ and $E_i\cdot F=E_j\cdot H=H^2=F^2=0$. 
\smallskip

 {\bf Chamber $\cb_1$}: The universal curve $\pi:\mathcal{X}_1\longrightarrow S$ for the chamber $\cb_1$ is obtained from $\mathcal{X}_0$ by contracting $E_3$. This is shown in Figure~\ref{fig:hascom}. Hence we obtain
  \[\frac{1}{2\pi}[\Omega_{\mathcal{X}_1/S}]=K_{\mathcal{X}_1}\otimes \pi^*(K_S^*)=-2H+E_1+E_2,
\]
hence
\begin{align}
    V_{0,\cb_1}(i\b{\theta})&=%\frac12([\Omega_{X/S}]+2\pi\mathbf{a}\cdot D)^2 \nonumber\\
%&=2\pi^2\Big((-2+\sum_{j=1}^4 a_j)H+a_4F+\sum_{j=1}^2(1-a_j-a_4)E_j\Big)^2 \nonumber\\
%&=2\pi^2\Big(2a_4(-2+\sum_{j=1}^4 a_j)-\sum_{j=1}^2(1-a_j-a_4)^2\Big) \nonumber\\
%&=2\pi^2\Big(-(1-a_1)^2-(1-a_2)^2+2a_3a_4\Big) \nonumber\\
%&=
-\frac12\theta_1^2-\frac12\theta_2^2+(2\pi-\theta_3)(2\pi-\theta_4). \label{C1vol}
\end{align}

  {\bf Chamber $\cb_2$}: The universal curve $\pi:\mathcal{X}_2\longrightarrow S$ for the chamber $\cb_2$ is obtained from $\mathcal{X}_1$ by contracting $E_2$.
  Hence we obtain \[\frac{1}{2\pi}[\Omega_{\mathcal{X}_2/S}]=K_{\mathcal{X}_2}\otimes \pi^*(K_S^*)=-2H+E_1,
\]
hence
\begin{align}
    V_{0,\cb_2}(i\b{\theta})&=%\frac12([\Omega_{X/S}]+2\pi\mathbf{a}\cdot D)^2 \nonumber\\
%&=2\pi^2\Big((-2+\sum_{j=1}^4 a_j)H+a_4F+(1-a_1-a_4)E_j\Big)^2 \nonumber\\
%&=2\pi^2\Big(2a_4(-2+\sum_{j=1}^4 a_j)-(1-a_1-a_4)^2\Big) \nonumber\\
%&=2\pi^2\Big(-(1-a_1)^2+2a_2a_4+2a_3a_4+(1-a_4)^2-1\Big) \nonumber\\
%&=
-\frac12\theta_1^2+\frac12\theta_4^2+(2\pi-\theta_4)(4\pi-\theta_2-\theta_3)-2\pi^2. \label{C2vol}
\end{align}

   {\bf Chamber $\cb_3$}:  The universal curve $\pi:\mathcal{X}_3\cong \bp^1\times\bp^1\longrightarrow S$ for the chamber $\cb_3$ is obtained from $\mathcal{X}_2$ by contracting $E_1$. Hence we obtain \[\frac{1}{2\pi}[\Omega_{\mathcal{X}_3/S}]=K_{\mathcal{X}_3}\otimes \pi^*(K_S^*)=-2H,
\]
hence
\begin{align}
    V_{0,\cb_3}(i\b{\theta})&=%\frac12([\Omega_{X/S}]+2\pi\mathbf{a}\cdot D)^2 \nonumber\\
%&=2\pi^2\Big((-2+\sum_{j=1}^4 a_j)H+a_4F\Big)^2 \nonumber\\
%&=2\pi^2\Big(2a_4(-2+\sum_{j=1}^4 a_j)\Big) \nonumber\\
%&=2\pi^2\Big(2(a_1+a_2+a_3)a_4+2(1-a_4)^2-2\Big) \nonumber\\
%&=
(2\pi-\theta_4)(4\pi-\sum_{j=1}^4\theta_j). \label{C3vol}
\end{align}

    {\bf Chamber $\cb_4$}: Consider the universal curve for the $\cb_2$ chamber
    $\pi:\mathcal{X}_2\longrightarrow S.    $
  The fibre over $p_1\in S$ has two irreducible components. One component is $E_1$, the other we label $F_1$, which is the proper transform of the fibre $\bp^1\times\{p_1\}$ in $\bp^1\times \bp^1$. The universal curve $\pi:\mathcal{X}_4\longrightarrow S$ for the chamber $\cb_4$ is obtained from $\mathcal{X}_2$ by contracting $F_1$. Hence we obtain
   \[\frac{1}{2\pi}[\Omega_{\mathcal{X}_4/S}]=K_{\mathcal{X}_4}\otimes \pi^*(K_S^*)=-2H+E_1,
\]
hence
\begin{align}
    V_{0,\cb_4}(i\b{\theta})&=\tfrac12(4\pi-\sum_{j=1}^4 \theta_j)(4\pi-\sum_{j=2}^4\theta_j+\theta_1).  \label{C4vol}
\end{align}

{\bf Comparison with Theorem~\ref{wall-crossing formula}}: The difference between volume polynomials in adjacent chambers gives the following wall-crossing polynomials
\begin{eqnarray*}
   \wc_{\cb_0,\{3,4\}}(\boldsymbol{\theta})&=&\eqref{C1vol}-\eqref{C0vol}=\frac12(\theta_3+\theta_4-2\pi)^2,\\
   \wc_{\cb_1,\{2,4\}}(\boldsymbol{\theta})&=&\eqref{C2vol}-\eqref{C1vol}=\frac12(\theta_2+\theta_4-2\pi)^2,\\ 
    \wc_{\cb_2,\{1,4\}}(\boldsymbol{\theta})&=&\eqref{C3vol}-\eqref{C2vol}=\frac12(\theta_1+\theta_4-2\pi)^2,\\
  \wc_{\cb_2,\{2,3\}}(\boldsymbol{\theta})&=&\eqref{C4vol}-\eqref{C2vol}=\frac12(\theta_2+\theta_3-2\pi)^2.
\end{eqnarray*}
Observe that Theorem~\ref{wall-crossing formula} gives \begin{equation*}
    \wc_{\cb_0,\{3,4\}}(\boldsymbol{\theta})=\int_{0}^{\theta_3+\theta_4-2\pi}\hspace{-0.3cm} V_{0,\cb_0/S}( i\theta_1,i\theta_2,i\theta)V_{0,\cc_1^L}( i\theta,i\theta_3,i\theta_4) \cdot \theta \cdot d\theta
    =\int_{0}^{\theta_3+\theta_4-2\pi}\hspace{-0.25cm} \theta \cdot d\theta.
\end{equation*}
Hence $$ \wc_{\cb_0,\{3,4\}}(\boldsymbol{\theta})=\frac12(\theta_3+\theta_4-2\pi)^2,
%=\frac12\theta_3^2+\frac12\theta_4^2+(2\pi-\theta_3)(2\pi-\theta_4)-2\pi^2,
$$
which agrees with the computation above. The other wall-crossing polynomials follow similarly.

%%%%%%%%%%%%%%%%%%%%%%%%%%%
\subsection{$\overline{\modm}_{0,\bold{a}}$ for $n=5$ }\label{ex:dependentwc}
In this section we compute all wall-crossing polynomials in this case using Theorem~\ref{wall-crossing formula} and the volumes computed in \S\ref{ex:M04}. There are two cases to consider.
\\

{\bf Crossing a wall $W_S$ for $|S|=2$:} without loss of generality we assume $S=\{4,5\}$.  Let $\cc\subset\cd_{0,5}$ be incident to and above a wall $W_S$.  Theorem~\ref{wall-crossing formula} then gives
\begin{equation*} 
\wc_{\cc,S}(\boldsymbol{\theta})=\int_{0}^{\phi_S} V_{0,\cc/S}(i\theta, i\theta_1,i\theta_2,i\theta_3)\cdot 1 \cdot \theta \cdot d\theta
\end{equation*}
where $\phi_S=\sum_{j\in S}\theta_j-2\pi$ and the chamber $\cc/S$ is dependent on the initial chamber $\cc$ and gives four cases specified by chambers in \S\ref{ex:M04} as
\begin{eqnarray*}
\text{Case 1:}&\cb_0& \text {for }\theta_i+\theta_j<2\pi,\text{ for all $i,j\in\{1,2,3\}$},\\
\text{Case 2:}&\cb_1& \text{for }\theta_1+\theta_2<2\pi,\hspace{0.2cm}\theta_1+\theta_3<2\pi,\hspace{0.2cm}\theta_2+\theta_3>2\pi,\\
\text{Case 3:}&\cb_2& \text{for }\theta_1+\theta_2<2\pi,\hspace{0.2cm}\theta_1+\theta_3>2\pi,\hspace{0.2cm}\theta_2+\theta_3>2\pi,\\
\text{Case 4:}&\cb_4& \text {for }\theta_i+\theta_j>2\pi,\text{ for all $i,j\in\{1,2,3\}$}.
\end{eqnarray*}
Hence we obtain the wall-crossing polynomial $\wc_{\cc,S}(\boldsymbol{\theta})$ becomes in Case 1:
$$\tfrac 18(2\pi-\theta_4-\theta_5)^2(4\pi^2-2(\theta_1^2+\theta_2^2+\theta_3^2)+4\pi(\theta_4+\theta_5)-(\theta_4+\theta_5)^2),  $$
in Case 2:
$$ \tfrac 18(2 \pi - \theta_4 - \theta_5)^2(12\pi^2-2\theta_1^2+4\theta_2\theta_3-(\theta_4+\theta_5)^2-4\pi(2\theta_2+2\theta_3-\theta_4-\theta_5)))),  $$
in Case 3:
$$ \tfrac 18(2 \pi - \theta_4 - \theta_5)^2(20\pi^2+2\theta_3(2\theta_1+2\theta_2+\theta_3)-(\theta_4+\theta_5)^2-4\pi(2\theta_1+2\theta_2+4\theta_3-\theta_4-\theta_5)), $$
and in Case 4:
$$\tfrac 18(2\pi-\theta_4-\theta_5)^2(2(4\pi-\theta_1-\theta_2-\theta_3)^3-(2\pi-\theta_4-\theta_5)^2).  $$

Observe that the wall-crossing polynomial $\wc_{\cc,S}(\boldsymbol{\theta})$ in this case is dependent on the incident chamber $\cc$. 
\\

{\bf Crossing a wall $W_S$ for $|S|=3$:} without loss of generality we assume $S=\{3,4,5\}$. Let $\cc\subset\cd_{0,5}$ be incident to and above a wall $W_S$.  Theorem~\ref{wall-crossing formula} then gives
\begin{equation*} 
\wc_{\cc,S}(\boldsymbol{\theta})=\int_{0}^{\phi_S} 1\cdot V_{0,\cb_4}(i\theta,i\theta_3,i\theta_4,i\theta_5)  \cdot \theta \cdot d\theta
\end{equation*}
where $\phi_S=\sum_{j\in S}\theta_j-4\pi$ and $\cb_4$ is as specified in \S\ref{ex:M04}. Hence we have
\begin{eqnarray*} 
\wc_{\cc,S}(\boldsymbol{\theta})&=&\int_{0}^{\phi_S} \tfrac12(4\pi-\theta-\sum_{j=3}^5 \theta_j)(4\pi+\theta-\sum_{j=3}^5\theta_j) \cdot \theta \cdot d\theta\\
&=&\tfrac 18(\theta_3+\theta_4+\theta_5-4\pi)^4
\end{eqnarray*}
Note that in this case the wall-crossing polynomial is independent of the choice of incident chamber $\cc$. This agrees with Corollary~\ref{cor:quotchambercrossing} as in this case all quotient chambers $\cc/S$ lie in $\cd_{0,3}$ and are hence equal.

%%%%%%%%%%%%%%%%%%%%%%%%%%%
%%%%%%%%%%%%%%%%%%%%%%%%%%

%%%%%%%%%%%%%%%%%%%%%%%%%%%%%%
%%%%%%%%%%%%%%%%%%%%%%%%%%%%%%

\end{document}